\documentclass{article}

\PassOptionsToPackage{numbers, compress}{natbib}
\usepackage[preprint]{neurips_2022}

\usepackage{amsmath}
\usepackage{amssymb}
\usepackage{mathtools}
\usepackage{amsthm}
\usepackage[utf8]{inputenc}
\usepackage[T1]{fontenc}
\usepackage{hyperref}
\usepackage{url}
\usepackage{booktabs}
\usepackage{amsfonts}
\usepackage{nicefrac}
\usepackage{microtype}
\usepackage{xcolor}

\usepackage{hyperref}
\usepackage[capitalize,noabbrev]{cleveref}
\usepackage{subfigure}
\usepackage{bm}
\usepackage{natbib}
\usepackage{algorithm}
\usepackage{algorithmic}

\DeclarePairedDelimiter\set{\{}{\}}
\DeclarePairedDelimiterX\Set[2]{\{}{\}}{\mspace{2mu}{#1}\;\delimsize|\;{#2}\mspace{2mu}}
\usepackage{multicol}
\usepackage{overpic}
\usepackage{extarrows}
\usepackage{siunitx}
\usepackage{physics}
\usepackage{color}
\usepackage{here}
\usepackage{comment}
\def\diag{\mathop{\rm diag}\nolimits}
\allowdisplaybreaks

\theoremstyle{plain}
\newtheorem{theorem}{Theorem}[section]

\newtheorem{lemma}[theorem]{Lemma}

\theoremstyle{definition}
\newtheorem{definition}[theorem]{Definition}
\newtheorem{assumption}[theorem]{Assumption}
\theoremstyle{remark}

\title{B\'ezier Flow: a Surface-wise Gradient Descent Method for Multi-objective Optimization}

\author{%
    Akiyoshi Sannai\\
    RIKEN AIP \\
    \texttt{akiyoshi.sannai@riken.jp}\\
    \And
    Yasunari Hikima\\
    Fujitsu Limited \\
    \texttt{hikima.yasunari@fujitsu.com}\\
    \And
    Ken Kobayashi\\
    Tokyo Institute of Technology\\
    \texttt{kobayashi.k.ar@m.titech.ac.jp}\\
    \And
    Akinori Tanaka\\
    RIKEN AIP \\
    \texttt{akinori.tanaka@riken.jp}
    \And
    Naoki Hamada\\
    KLab Inc.\\
    \texttt{hamada-n@klab.com}
}

\begin{document}

\maketitle

\begin{abstract}
In this paper, we propose a strategy to construct a multi-objective optimization algorithm from a single-objective optimization algorithm by using the B\'ezier simplex model. 
Also, we extend the stability of optimization algorithms in the sense of Probability Approximately Correct (PAC) learning and define the PAC stability. 
We prove that it leads to an upper bound on the generalization with high  probability.
Furthermore, we show that multi-objective optimization algorithms derived from a gradient descent-based single-objective optimization algorithm are PAC stable.
We conducted numerical experiments and demonstrated that our method achieved lower generalization errors than the existing multi-objective optimization algorithm.
\end{abstract}

\section{Introduction}
A multi-objective optimization problem is a problem to seek a solution which minimizes (or maximizes) multiple objective functions $f_1,\dots,f_M:X\to\mathbb{R}$ simultaneously over a domain $X\subseteq\mathbb{R}^L$:
\begin{align*}
    \mathrm{minimize}\quad &\bm{f}(\bm{x})\coloneqq(f_1(\bm{x}),\dots,f_M(\bm{x}))^\top \\
    \mathrm{subject\:to}\quad &\bm{x}\in X\subseteq \mathbb{R}^L.
\end{align*}
Each objective function can have a different optimal solution, so we need to consider the trade-off between two or more solutions. Therefore, the notion of Pareto ordering is taken into consideration which is defined by
\begin{align*}
    \bm{f}(\bm{x}) \prec \bm{f}(\bm{y})  \overset{\text{def}}{\Longleftrightarrow}
    &f_m(\bm{x}) \leq f_m(\bm{y}) \text{ for all $m=1,\dots,M,$} \\
    &\text{and } f_m(\bm{x}) < f_m(\bm{y}) \text{ for some $m=1,\dots,M$}.
\end{align*}
In multi-objective optimization, the goal is to obtain the Pareto set and Pareto front, which are respectively defined as:
\begin{align*}
    X^\star(\bm{f}) \coloneqq \qty{\bm{x}\in X \mid f(\bm{y}) \nprec f(\bm{x}) \text{ for all $\bm{y} \in X$}},~
    \bm{f}X^\star(\bm{f}) \coloneqq \qty{\bm{f}(\bm{x})\in\mathbb{R}^M\mid\bm{x}\in X^\star(\bm{f})}.
\end{align*}

The Pareto set/front usually has an infinite number of points, whereas most of the numerical methods for solving the problem give us a finite set of points as an approximation of the Pareto set/front (e.g., goal programming \citep{Miettinen1999,Eichfelder2008}, evolutionary computation \citep{Deb2001,Zhang2007,Deb2014}, homotopy methods \citep{Hillermeier2001,Harada2007}, and Bayesian optimization \citep{Hernandez-Lobato2016,Yang2019}).
Such a finite-point approximation cannot reveal the complete shape of the Pareto set and front.
In addition, the finite-point approximation suffers from the ``curse of dimensionality'' since the dimensionality of the Pareto set and front is $M-1$ in generic problems (see \citet{Wan1977,Wan1978} for rigorous statement).
Again this background, we consider in this paper an optimization algorithm to obtain a parametric hypersurface describing the Pareto set.

There is a common structure of the Pareto set/front across a wide variety of problems, which can be utilized to enhance approximation.
In many problems, obtained solutions imply the Pareto set/front is a curved $(M-1)$-simplex, e.g., airplane design \citep{Mastroddi2013}, hydrologic modeling \citep{Vrugt2003}, PI controller tuning \citep{Reynoso-Meza2015}, building design \citep{Gilan2016}, motor design \citep{Contreras2016}, and lasso's hyper-parameter tuning \citep{Hamada2020b}.
To mathematically identify such a class of problems, \citet{Kobayashi2019} defined the \emph{simplicial} problem (see \Cref{fig:face-relation}).
\citet{Hamada2020} showed that strongly convex problems are simplicial under mild conditions, which implies facility location \citep{Kuhn1967} and phenotypic divergence modeling in evolutionary biology \citep{Shoval2012} are simplicial.
\citet{Kobayashi2019} showed that the Pareto set and front of any simplicial problem can be approximated with arbitrary accuracy by a B\'ezier simplex.

\begin{figure*}[t]
    \centering
    \includegraphics[width=0.70\linewidth]{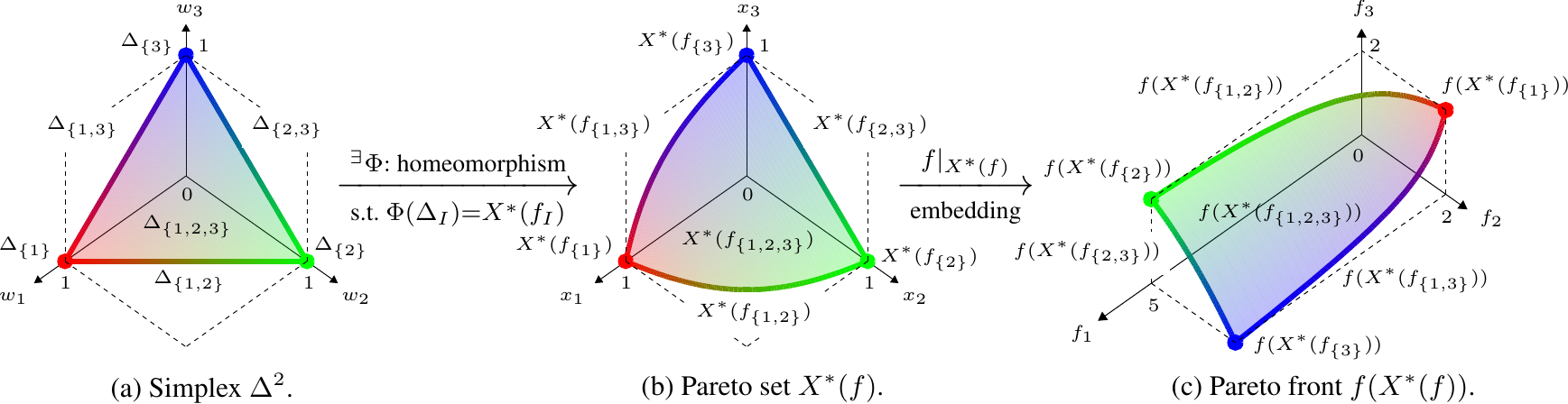}
    \caption{A simplicial problem $f = (f_1, f_2, f_3)^\top: \mathbb{R}^3 \to \mathbb{R}^3$. An $M$-objective problem $f$ is \emph{simplicial} if the following conditions are satisfied: (i) there exists a homeomorphism $\Phi: \Delta^{M - 1} \to X^*(f)$ such that $\Phi(\Delta_I) = X^*(f_I)$ for all $I \subseteq \{1, \dots, M\}$; (ii) the restriction $f|_{X^*(f)}: X^*(f) \to \mathbb{R}^M$ is a topological embedding (and thus so is $f \circ \Phi: \Delta^{M - 1} \to \mathbb{R}^M$).}\label{fig:face-relation}
\end{figure*}

By using this advantage of the B\'ezier simplex model, we propose a novel strategy to construct a multi-objective optimization algorithm from a single-objective optimization. 
With a given single objective optimization algorithm, this scheme updates the B\'ezier simplex to obtain the Pareto set.
In addition, we analyze the theoretical property of the multi-objective optimization algorithm derived from our scheme.
Specifically, we first define Probably Approximately Correct (PAC) stability as an extension of the stability of optimization algorithms and prove that the PAC stability leads to an upper bound on the generalization gap in the sense of PAC learning. 
Our contributions are summarized as follows:
\begin{enumerate}
    \item We devise a strategy to construct a multi-objective optimization algorithm from a single-objective optimization algorithm with the B\'ezier simplex. 
    Unlike most of the existing multi-objective optimization methods, the algorithm derived from our scheme has the advantage of obtaining a parametric hyper-surface that represents the Pareto set of a given simplicial, Lipschitz continuous, differentiable multi-objective optimization problem to be solved.
    \item We define PAC stability, which is an extension of the stability introduced by \citet{hardt2016train} to the PAC learning settings and show that PAC stability gives an upper bound on the generalization gap with a high probability. 
    Also, we prove that when we employ a gradient-based optimization algorithm as a single optimization algorithm, the derived multi-objective optimization algorithm is PAC stable. 
    \item We conducted numerical experiments and demonstrated that the multi-objective optimization algorithm constructed by our scheme achieved lower generalization errors than the existing multi-objective optimization algorithm. In addition, the algorithm given by our scheme can efficiently obtain the Pareto set with a small number of sample points.
\end{enumerate}
\paragraph{Related Work}
\citet{Kobayashi2019} proposed B\`ezier simplex fitting algorithms, the all-at-once fitting, and inductive skeleton fitting to describe Pareto fronts, and \citet{Tanaka2020} analyzed the asymptotic risk of the fitting algorithms. 
The two fitting algorithms focus on post-optimization processes and assume that we have an approximate solution set of the Pareto set in advance. 
Thus, these algorithms by themselves cannot solve multi-objective optimization problems. 
Recently, \citet{Maree2020} proposed a bi-objective optimization algorithm that updates the B\'ezier curve. 
However, this algorithm exploits the structure of the bi-objective optimization problem and can not be applied when the number of objective functions is more or equal to three. 
To the best of our knowledge, we are the first to propose a general framework of multi-objective optimization with the B\'ezier simplex and show its theoretical property. 
\section{Preliminaries}
\subsection{Probability simplex}\label{prob sim}
Let $[M]=\{1,\dots,M\}$ be a set of $M$ points.
We consider the set of probability distribution $\bm{t}$ over $[M]$. 
The set of probability distributions over $[M]$ is equal to the simplex
\footnotesize
\begin{align*}
\Delta^{M-1} \coloneqq
    \Set*{(t_1,\dots,t_M)^\top\in\mathbb{R}^M}{t_m\geq 0,\, \sum_{m=1}^{M} t_m = 1}.
\end{align*}
\normalsize
Let $C(X)$ be the space of continuous functions over $X$, and we define the function $F\colon [M] \to C(X)$ by $F(m)=f_m$. 
Then, we have the expectation function
\begin{align*}
\begin{array}{rccc}
\mathbb{E}(\bm{f})\colon &\Delta                      &\longrightarrow& C(X)                     \\
        & \rotatebox{90}{$\in$}&               & \rotatebox{90}{$\in$} \\
        & \bm{t}                   & \longmapsto   & \mathbb{E}_{\bm{t}}(F)
\end{array}.
\end{align*}

Furthermore, if $f_m$ is strongly convex for all $m \in [M]$, then the following function is well-defined:
\begin{align*}
\begin{array}{rccc}
\arg\min\mathbb{E}(\bm{f})\colon &\Delta                      &\longrightarrow& X                \\
        & \rotatebox{90}{$\in$}&               & \rotatebox{90}{$\in$} \\
        & \bm{t}                   & \longmapsto   & \arg\min\mathbb{E}_{\bm{t}}(F)
\end{array}.
\end{align*}

Note that $\mathbb{E}_{\bm{t}}(F)=\sum_m t_m f_m$ follows from the definition. $\mathbb{E}_{\bm{t}}(F)$ corresponds to the sum of a function chosen continuously along $\bm{t}$ from $\bm{f}$. 
As a direct consequence from Theorem 2 in \citet{Mizota2021}, the mapping $\arg\min\mathbb{E}(\bm{f})$ gives a continuous surjection onto $X^\star(\bm{f})$ if $f_m$ is strongly convex for all $m \in [M]$.

\subsection{Simplicial Problem}
A multi-objective optimization problem is characterized by its objective map $\bm{f}=(f_1,\dots,f_M)^\top\colon X\to\mathbb{R}^M$. 
We define the \emph{$J$-subsimplex} for an index set $J \subseteq [M]$ by $\Delta^{M - 1}_J \coloneqq \set{(t_1, \dots, t_M)^\top \in \Delta^{M - 1} \mid t_m = 0\ (m \not \in J)}$.
The problem class we wish to consider is a problem in which the Pareto set/front has the simplex structure. Such problem class is defined as follows.

\begin{definition}[\citet{Kobayashi2019}]
A problem $\bm{f}\colon X\rightarrow\mathbb{R}^M$ is \emph{simplicial} if there exists a map $\bm{\phi}\colon\Delta^{M-1}\rightarrow X$ such that for each non-empty subset $J\subseteq [M]$, its restriction $\bm{\phi}|_{\Delta^{(M-1)}_J}\colon\Delta^{M-1}_J\rightarrow X$ gives homeomorphisms
\begin{align*}
    \bm{\phi}|_{\Delta^{M-1}_J} \colon \Delta^{M-1}_J\to X^\star (\bm{f}_J),\quad
    \bm{f}\circ\bm{\phi}|_{\Delta^{M-1}_J} \colon \Delta^{M-1}_J\to \bm{f} X^\star (\bm{f}_J).
\end{align*}
We call such $\bm{\phi}$ and $\bm{f}\circ\bm{\phi}$ a \emph{triangulation} of the Pareto set $X^\star (\bm{f})$ and the Pareto front $\bm{f}X^\star (\bm{f})$, respectively.
\end{definition}

\subsection{B\'ezier Simplex}
Let $\mathbb{N}$ be the set of nonnegative integers and
\begin{align*}
    \mathbb{N}^M_D \coloneqq \Set*{(d_1,\dots,d_M)^\top\in\mathbb{N}^M}{\sum_{m=1}^M d_m = D}.
\end{align*}
For $\bm{t}\coloneqq(t_1,\dots,t_M)^\top\in\Delta^{M-1}$ and $\bm{d}\coloneqq(d_1,\dots,d_M)^\top\in \mathbb{N}^M_D$, we denote by $\bm{t}^{\bm{d}}$ a monomial $t_1^{d_1}\dots t_M^{d_M}$. The B\'ezier simplex of degree $D$ in $\mathbb{R}^L$ with control points $\{\bm{p}_{\bm{d}}\}_{\bm{d}\in\mathbb{N}^M_D}$ is defined as a map $\bm{b}\colon\Delta^{M-1}\to\mathbb{R}^L$:
\begin{align}\label{eq:b}
    \bm{b}(\bm{t}|\bm{P})
    \coloneqq \sum_{\bm{d}\in\mathbb{N}^M_D} \binom{D}{\bm{d}} \bm{t}^{\bm{d}}\bm{p}_{\bm{d}},
\end{align}
where $\binom{D}{\bm{d}}$ is a multinomial coefficient and $\bm{P}$ represents a matrix of vertically aligned control points:
\begin{align*}
    \bm{P} \coloneqq \qty(\bm{p}_1,\dots,\bm{p}_{|\mathbb{N}^M_D|})^\top\in\mathbb{R}^{|\mathbb{N}^M_D|\times L}.
\end{align*}

We define $\bm{z}(\bm{t})$ as
\begin{align*}
    \bm{z}(\bm{t}) \coloneqq\qty[ \binom{D}{\bm{d}_1}\bm{t}^{\bm{d}_1},\dots,
    \binom{D}{\bm{d}_{|\mathbb{N}^M_D|}}\bm{t}^{\bm{d}_{|\mathbb{N}^M_D|}}]^\top \in \mathbb{R}^{|\mathbb{N}^M_D|}.
\end{align*}
Then, \eqref{eq:b} can be represented as $\bm{b}(\bm{t}|\bm{P}) = \bm{P}^\top\bm{z}(\bm{t})$. 
It is known that B\'ezier simplex is a universal approximator of continuous functions \cite{Kobayashi2019}, and thus, the mapping $\arg\min\mathbb{E}(\bm{f})$ can be approximated by B\'ezier simplices in arbitrary precision.
From this theoretical advantage, we construct a general framework to obtain a multi-objective optimization method $\mathcal{M}(A)$ from a single-objective optimization method $A$ with B\'ezier simplices.

\section{Proposed Algorithm}
\begin{figure*}[t]
    \centering
    \subfigure[Generate solutions on the \newline B\'ezier simplex.]{
        \begin{overpic}[width=0.33\linewidth,trim=30 30 30 30,clip]{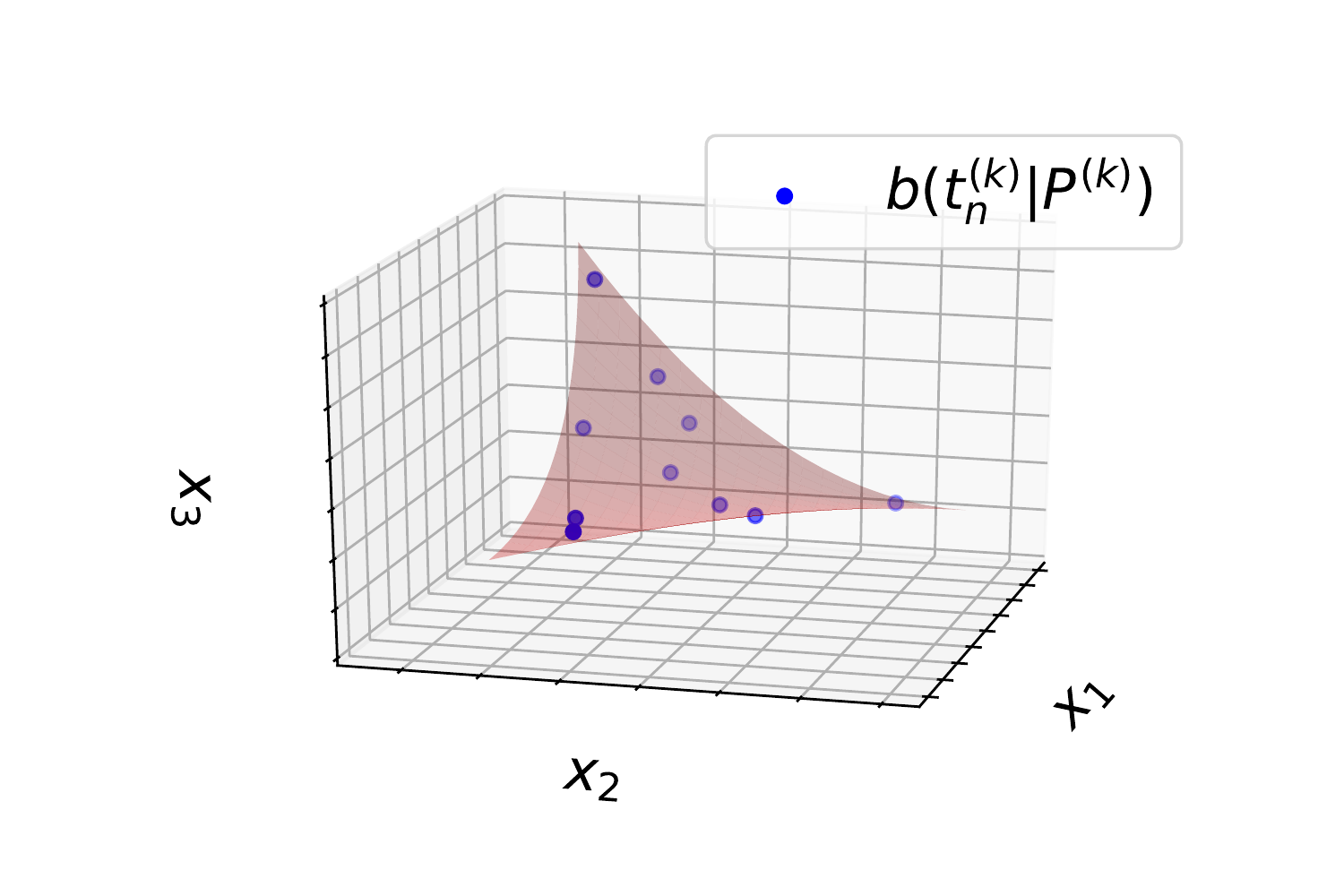}
            \put(85, 30){\large $\longrightarrow$}
        \end{overpic}}%
    \subfigure[Update each solution with~$A_{\bm{t}^{(k)}_n}$.]{
        \begin{overpic}[width=0.33\linewidth,trim=30 30 30 30,clip]{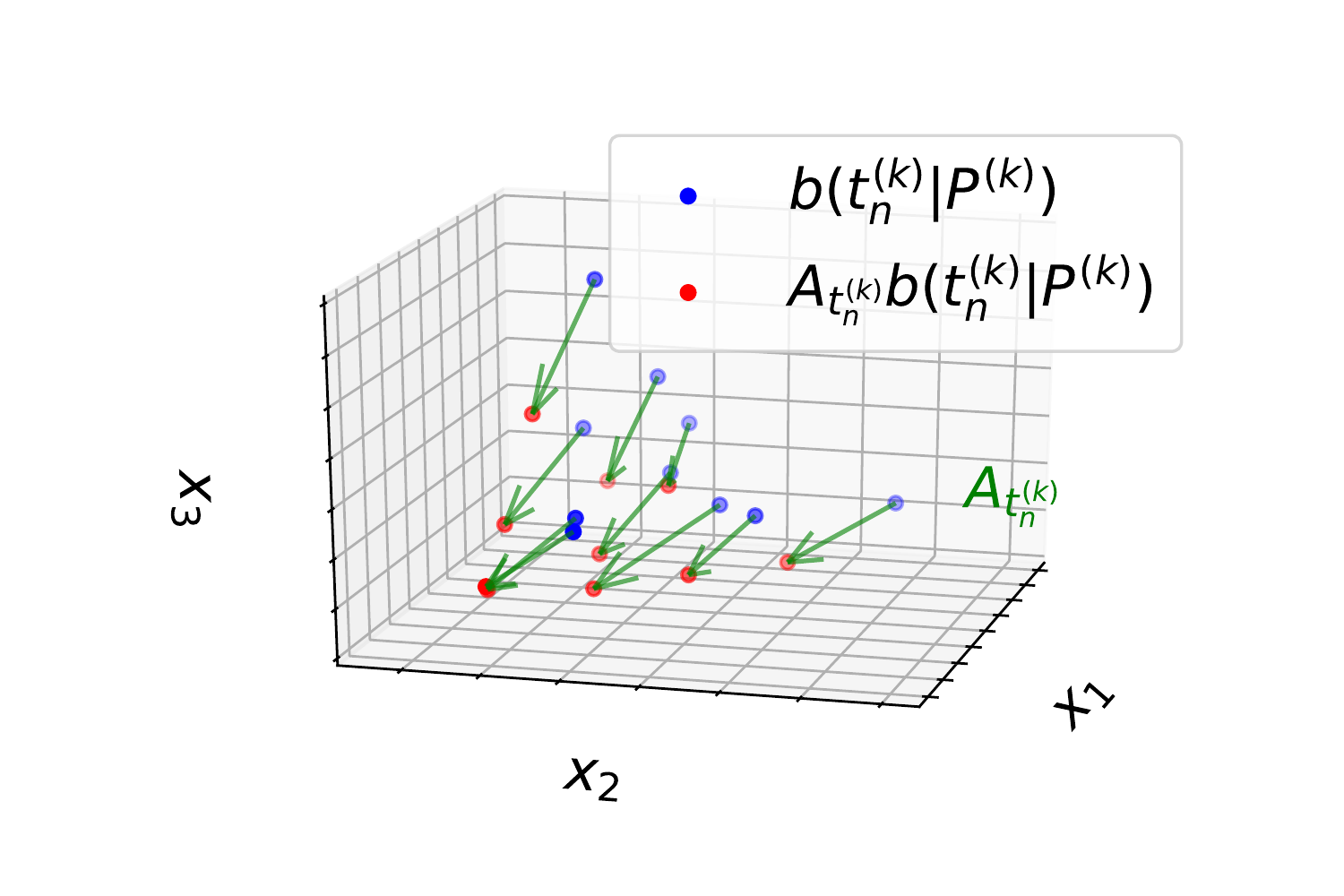}
            \put(85, 30){\large $\longrightarrow$}
        \end{overpic}}%
    \subfigure[Update the B\'ezier simplex.]{\includegraphics[ width=0.33\linewidth,trim=30 30 30 30,clip]{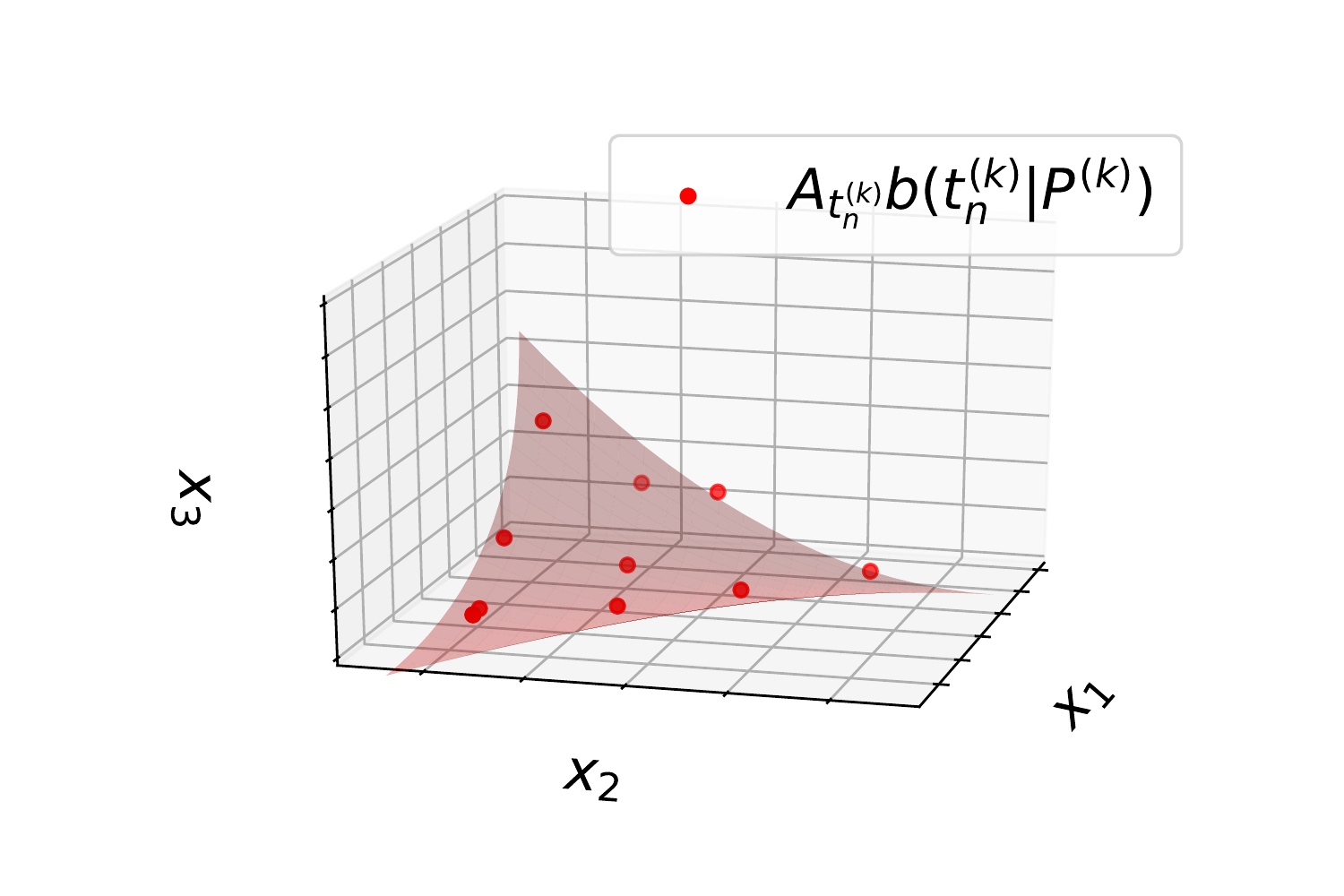}}%
    \caption{Conceptual diagram of $\mathcal{M}(A)$ at the $k$th iteration; the red surfaces in (a) and (c) represent the B\`ezier simplices.}
    \label{fig:ma}
\end{figure*}

\begin{algorithm}[ht]
    \caption{Multi-objective Optimization Method $\mathcal{M}(A)$ of the Optimization Method $A$}
    \label{alg:multi-op}
\begin{algorithmic}[1]
   \STATE Set $k\leftarrow 1$ and the initial control point $\bm{P}^{(k)}$.
   \WHILE{$k\le K$}
   \STATE Draw $\{\bm{t}^{(k)}_n\}_{n=1}^{N}$ for which each $\bm{t}^{(k)}_n$ is drawn i.i.d. from the uniform distribution on $\Delta^{M-1}$.
   \STATE Obtain $\{\bm{b}(\bm{t}^{(k)}_n|\bm{P}^{(k)})\}_{n=1}^{N}$  by (\ref{eq:b}).
   \STATE Obtain $\{A_{\bm{t}^{(k)}_n}(\bm{b}(\bm{t}^{(k)}_n|\bm{P}^{(k)}))\}_{n=1}^{N}$ from $A$.
   \STATE Update control points by (\ref{update_ma}).
   \STATE $k \leftarrow k + 1$.
   \ENDWHILE
   \RETURN $\bm{P}^{(K+1)}$.
\end{algorithmic}
\end{algorithm}

A number of methods have been studied in the context of  multi-objective optimization. Many of the methods are designed to apply to any multi-objective optimization problems; however, the individual methods are written in separate contexts and are not unified. Therefore, in this paper, we introduce a general framework to obtain a multi-objective optimization method $\mathcal{M}(A)$ from a single-objective optimization algorithm $A$. 
Moreover, in contrast to the existing methods finding a finite set that approximates Pareto set/front, we obtain a parametric hypersurface representing the Pareto set/front of multi-objective problems.

In our proposed algorithm, we obtain control points of a B\'ezier simplex that represents the Pareto set of a problem to be solved from a single-objective optimization algorithm $A$. In this paper, a single-objective optimization algorithm $A$ is a map from the direct product of the sample space and the space of loss functions $Z \times \mathcal{L}$ to the space of model parameters $\mathcal{W}$. Then, for any $\bm{t} \in \Delta^{M-1}$, we denote  by $A_{\bm{t}}$ a single-objective optimization algorithm $A$ with the loss function  $\mathbb{E}_{\bm{t}}(F)$, i.e., $A(-, \mathbb{E}_{\bm{t}}(F))$. 

Our algorithm begins by setting the initial control points $\bm{P}^{(1)}$.
At the $k$th iteration~($k\ge 1$), we randomly sample $\{\bm{t}^{(k)}_n\}^N_{n=1}$ from the uniform distribution on $\Delta^{M-1}$ and obtain data points $\{\bm{b}(\bm{t}^{(k)}_n|\bm{P}^{(k)})\}_{n=1}^N$ on the current B\'ezier simplex.
Next, we update each $\bm{b}(\bm{t}^{(k)}_n|\bm{P}^{(k)})$ by $A_{\bm{t}^{(k)}_n}$, i.e.,
\begin{align}\label{update_A}
\bm{x}^{(k)}_n=A_{\bm{t}^{(k)}_n}\qty(\bm{b}(\bm{t}^{(k)}_n|\bm{P}^{(k)})).
\end{align}
Then, we update the B\'ezier simplex with $\{(\bm{t}_n^{(k)} ,\bm{x}^{(k)}_n)\}_{n=1}^{N}$. 
Specifically, we solve the following optimization problem to fit a B\'ezier simplex to $\{(\bm{t}_n^{(k)} ,\bm{x}^{(k)}_n)\}_{n=1}^{N}$:
\begin{align}\label{opt:find_p}
    \underset{\bm{P}\in\mathbb{R}^{|\mathbb{N}^M_D|\times L}}{\mathrm{minimize}}\: &\frac{1}{N}\sum_{n=1}^{N} \left\|\bm{x}^{(k)}_n - \bm{b}(\bm{t}^{(k)
    }_n|\bm{P})\right\|^2_2,
\end{align}
where $\bm{P}$ is a variable to be optimized, and $\|\cdot\|_2$ denotes the Euclidean norm. 
Let $\bm{X}^{(k)}$ and $\bm Z^{(k)}$ be a matrix of vertically aligned $\bm{x}^{(k)}_n$ and $\bm z(\bm t^{(k)}_n)$, respectively. 
Then, the problem (\ref{opt:find_p}) is reformulated as
\begin{align}\label{opt:find_p_2}
    \underset{\bm{P}\in\mathbb{R}^{|\mathbb{N}^M_D|\times L}}{\mathrm{minimize}}\:
    \frac{1}{N}\left\| \bm{X}^{(k)} - \bm{Z}^{(k)}\bm{P} \right\|^2_F,
\end{align}
where $\|\cdot\|_F$ denotes the Frobenius norm.
Since the optimization problem (\ref{opt:find_p_2}) is an unconstrained convex quadratic optimization, and it can be shown that $\bm{Z}^{(k)\top}\bm{Z}^{(k)}$ is regular with probability 1, the update rule for control points is described as
\begin{align}\label{update_ma}
\bm{P}^{(k+1)} = \qty(\bm{Z}^{(k)\top}\bm{Z}^{(k)})^{-1} \bm{Z}^{(k)\top}\bm{X}^{(k)}.
\end{align} 
We repeat this procedure until $k$ reaches the maximum number of iterations specified by the user.
We summarize the multi-objective optimization method in \Cref{alg:multi-op} and show its conceptual diagram in \Cref{fig:ma}. 

\section{PAC Stability and Generalization Gap}\label{sec:gere}

Assume that there is an unknown distribution $\mathcal{D}$ over some space $Z$. We take  $S=\left(\bm{t}_{1}, \ldots, \bm{t}_{N}\right)$ of $N$ examples drawn i.i.d. from $\mathcal{D}$. Then the generalization error is defined by:
$$
R[\bm{P}] \stackrel{\text { def }}{=} \mathbb{E}_{\bm{t} \sim \mathcal{D}} \ell(\bm{P} ; \bm{t}),
$$
where $\ell$ is a loss function, and $\ell(\bm{P} ; \bm{t})$ denotes the loss of the model described by $\bm{P}$ with an input $\bm{t}$.
Since the generalization error cannot be measured directly, we instead consider the empirical error defined by $R_{S}(\bm{P}) \coloneqq \frac{1}{N} \sum^N_{n=1} \ell(\bm{P};\bm{t}_n)$.
Then, the generalization gap of $\bm{P}$ is defined as the difference between empirical error and generalization error, i.e.,
\begin{align}\label{gen err}
    R_{S}(\bm{P}) - R(\bm{P}).
\end{align}
We consider a potentially randomized algorithm $A$ (e.g., stochastic gradient descent) and the expectation value of~(\ref{gen err}):
\begin{align}\label{exp_algo}
 \mathbb{E}_{A}\left[R_{S}[A(S)]-R[A(S)]\right].
\end{align}

To treat the approximate behavior of the expectation value with respect to the sample, we consider the following.
First, take an event $C \subset Z^N$ that has a high probability of occurring. Then, the conditional generalization error under the condition $C$ is defined by:
$$
\hat{R}[\bm{P}] \stackrel{\text { def }}{=} \mathbb{E}_{(\bm{t}_1,...,\bm{t}_N) \sim \mathcal{D}^N_{C}}\left [ \frac{1}{N}\sum_{i=1}^N \ell(\bm{P} ; \bm{t}_i)\right],
$$
where $\mathcal{D}^N_C$ is the conditional probability distribution of $C$. Note that if $C=Z^N$,  $\hat{R}[\bm{P}] $ is equal to $R[\bm{P}] $.

Next, we consider the approximate expectation value of (\ref{exp_algo}) by
\begin{align}\label{exp_algo_}
\hat{\mathbb{E}}_{S}\mathbb{E}_{A}\left[R_{S}[A(S)]-\hat{R}[A(S)]\right],
\end{align}
where $\hat{\mathbb{E}}_{S}$ is the conditional expectation value of $C$.
This invariant allows us to discuss the expectation value of the generalization gap with respect to events.

The following introduces the definition of \emph{probably approximately correct (PAC) uniform stability}. This is a PAC-like expansion of the uniform stability in \cite{hardt2016train}.
\begin{definition}\label{PAC unif}
A randomized algorithm $A$ is \emph{PAC uniformly stable} if for any $\varepsilon\in (0,1)$, there exists $\delta>0$ and an event $D_{\varepsilon} \subset Z^{N+1}$ which occurs with probability at least $1-\varepsilon $ such that
\begin{align}\label{PAC stable}
 \sup_{\bm{t}} \mathbb{E}_A\left[ \left| \ell(A(S);\bm{t})-\ell(A(S');\bm{t}) \right| \right] < \delta,
\end{align}
where $S=\left(\bm{t}_{1}, \ldots, \bm{t}_{N}\right)$ and $S^{\prime}=\left(\bm{t}_1,...,\bm{t}_{i}^{\prime}, \ldots, \bm{t}_{N}\right)$ are samples differing in at most one example, drawn from $\mathcal{D}$,  satisfying  $\left(\bm{t}_1,..., \bm{t}_i, \bm{t}_{i}^{\prime}, \bm{t}_{i+1} \ldots, \bm{t}_{N}\right) \in D_{\varepsilon}$. Furthermore, a \emph{PAC uniformly stable} randomized algorithm $A$ is decomposable if for any $\varepsilon\in (0,1)$, there are events $B_{\varepsilon} \subset Z$ such that $D_{\varepsilon}=B_{\varepsilon}^{N+1}$.

\end{definition}
With the PAC stability, we show the following that ensures that if an algorithm is PAC uniformly stable, the difference between its generalization and empirical error is small with high probability. 
\begin{theorem}\label{th:general}
Let $A$ be a decomposable PAC uniformly stable randomized algorithm. Then,  for any $\varepsilon\in (0,1)$ and $\delta>0$ in Definition \ref{PAC unif}, there exists an event $C_{\varepsilon} \subset Z^N$ which occurs with probability at least $1-\varepsilon$ such that,
$$
\left| \hat{\mathbb{E}}_{S}\mathbb{E}_{A}\left[R_{S}[A(S)]-\hat{R}[A(S)]\right] \right| < \delta,
$$
where $\hat{\mathbb{E}}_{S}$ is the conditional expectation value of $C_{\varepsilon}$ and $\hat{R}$ is the conditional generalization error under the condition $C_{\varepsilon}$.
\end{theorem}
The proof of \Cref{th:general} is shown in Appendix.
\section{A surface-wise gradient descent method}
Next, we discuss the case that $A$ is a gradient descent method. 
In this case, the update rule \eqref{update_A} can be represented as
\begin{align}
    \bm{x}^{(k)}_n &= \bm{b}(\bm{t}^{(k)}_n|\bm{P}) - \alpha^{(k)}\mathrm{d}_{\bm{x}}f\qty(\bm{b}(\bm{t}^{(k)
    }_n|\bm{P})\left|\, \bm{t}^{(k)}_n \right.) = \bm{b}(\bm{t}^{(k)}_n|\bm{P}) - \alpha^{(k)}J_{\bm{f}}\qty(\bm{b}(\bm{t}^{(k)}_n|\bm{P}))^\top\bm{t}^{(k)}_n,\label{eq:update_x}
\end{align}
where $\alpha^{(k)} \in (0,1]$ is a step size at the $k$th iteration, $\mathrm{d}_{\bm{x}}$ is a first derivative with respect to $\bm{x}$, $f(\cdot|\bm{t})$ is a weighted sum of objective functions $f_1,\dots,f_M$ by $\bm{t}$, and $J_{\bm{f}}(\bm{x})$ is a matrix of vertically aligned gradient of $f_m$ at $\bm{x}$ defined by $J_{\bm{f}}(\bm{x}) \coloneqq \qty(\nabla f_1(\bm{x}),\dots, \nabla f_M(\bm{x}))^\top\in\mathbb{R}^{M\times L}$.
Let us define $\bm{B}^{(k)}$ and $\bm{G}^{(k)}$ as
\begin{align*}
    \bm{B}^{(k)} \coloneqq \bm{Z}^{(k)}\bm{P}^{(k)}, \;
   \bm{G}^{(k)} \coloneqq \mqty[%
   (\bm{t}^{(k)}_1)^\top J_{\bm{f}}(\bm{P}^{(k)\top} \bm{z}_1)\\
   \vdots\\
   (\bm{t}^{(k)}_N)^\top J_{\bm{f}}(\bm{P}^{(k)\top} \bm{z}_N)]
\end{align*}
Then, the update rule (\ref{eq:update_x}) is rewritten as
\begin{align*}
    \bm{X}^{(k)} = \bm{B}^{(k)} - \alpha^{(k)}\bm{G}^{(k)}.
\end{align*}
With this notation, the update rule for the control points \eqref{update_ma} is represented as 
\begin{align}
    \bm{P}^{(k+1)} =\bm{P}^{(k)} - \alpha^{(k)}\qty(\bm{Z}^{(k)\top}\bm{Z}^{(k)})^{-1}\bm{Z}^{(k)\top}\bm{G}^{(k)}. \label{eq:update_p}
\end{align}
We describe the surface-wise gradient descent method in \Cref{alg:proposed}. 

\begin{algorithm}[ht]
   \caption{Surface-wise Gradient Descent Method}
   \label{alg:proposed}
\begin{algorithmic}[1]
   \STATE Set $k\leftarrow 1$ and the initial control point $\bm P^{(k)}$.
   \WHILE{$k\le K$}
   \STATE Draw $\{\bm{t}^{(k)}_n\}_{n=1}^{N}$ for which each $\bm{t}^{(k)}_n$ is drawn i.i.d. from the uniform distribution on $\Delta^{M-1}$.
   \STATE Obtain $\{\bm{b}(\bm{t}^{(k)}_n|\bm{P}^{(k)})\}_{n=1}^{N}$ by (\ref{eq:b}).
   \STATE Update $\{\bm{b}(\bm{t}^{(k)}_n|\bm{P}^{(k)})\}_{n=1}^{N}$ by (\ref{eq:update_x}).
   \STATE Update control points by  (\ref{eq:update_p}).
   \STATE $k \leftarrow k + 1$.
   \ENDWHILE
   \RETURN $\bm{P}^{(K+1)}$.
\end{algorithmic}
\end{algorithm}
\section{PAC Stability of the surface-wise gradient descent method}
We prove that the surface-wise gradient descent is PAC uniformly stable. 
All omitted proofs are shown in Appendix. 
Hereinafter, we make the following mild assumption about the objective function.
\begin{assumption}\label{assumption:mu}
All the objective functions $f_1,\dots,f_M$ are $\mu$-Lipschitz continuous and differentiable on $X$.
\end{assumption}

Let $\bm{x}^\star\colon\Delta^{M-1}\to X^\star(\bm{f})$ be a map from $\Delta^{M-1}$ to the Pareto set of $\bm{f}$.
For $\bm{P}$, we define a loss function as
\begin{align}\label{def:lossfunc}
    \ell(\bm{P};\bm{t}) \coloneqq \| \bm{b}(\bm{t}|\bm{P}) - \bm{x}^\star(\bm{t}) \|_2.
\end{align}
Since $X^*(\bm{f})$ is unknown, we can not take a sample directly from $X^\star(\bm{f})$. Instead, we take a sample $\bm{t}=\{\bm{t}_n\}^N_{n=1}$ drawn i.i.d. from the uniform distribution over $\Delta^{M-1}$. 

To prove that \Cref{alg:proposed} is PAC uniformly stable, we first show two propositions in advance.
Note that the following two propositions can respectively be regarded as an extension of the concept of boundedness and expansiveness introduced in \citep{hardt2016train} to analyze the stability of an optimization algorithm. 

\begin{lemma}\label{lemma:sigma-bounded}
Let $U>0$ be a constant satisfying $\max_{\bm{t}\in\Delta^{M-1}}\|\bm{z}(\bm{t})\|_2 \leq U$.
Let $\varphi_{\bm{T}}$ be the update rule with parameter $\bm{T}=\{\bm{t}_n\}^N_{n=1}$ in \eqref{eq:update_p}.
Then there exists some $\eta > 0$, and we have the following inequality with probability at least $1-\varepsilon$:
\begin{align*}
    \left\| \varphi_{\bm{T}}(\bm{P}) - \bm{P} \right\|_F \leq \eta NU\mu.
\end{align*}
\end{lemma}
\begin{lemma}\label{lemma:growth}
Let $\eta>0$ and $U>0$ be constants as in \Cref{lemma:sigma-bounded}.
For $\bm{T}=\{\bm{t}_n\}^N_{n=1}$ and $\bm{T}'=\{\bm{t}'_n\}^N_{n=1}$ such that the difference between $\bm{T}$ and $\bm{T}'$ lies only in one example, there exists some $\zeta > 0$, and we have the following with probability at least $1-\varepsilon$:
\begin{align*}
    \left\|\varphi_{\bm{T}}(\bm{P}) - \varphi_{\bm{T}'}(\bm{P})\right\|_F \leq \mu U \qty(\eta + \zeta N).
\end{align*}
\end{lemma}

Let $\{\bm{T}_i\}^{K}_{i=1}$ and $\{\bm{T}'_i\}^{K}_{i=1}$ be parameters whose difference lies only in the $k$th element, and $\bm P^{(K+1)}$ and $\bm P'^{(K+1)}$ be respectively the output of \Cref{alg:proposed} with $\{\bm{T}_i\}^{K}_{i=1}$ and $\{\bm{T}'_i\}^{K}_{i=1}$.
From \Cref{lemma:sigma-bounded,,lemma:growth}, we show that $\|\bm P^{(K+1)}-\bm P'^{(K+1)}\|_F$ is bounded above with arbitrary probability.
\begin{lemma}\label{lemma:P-gap}
Let $U>0,\,\eta>0$ and $\zeta>0$ be constants as in \Cref{lemma:growth}.
Suppose that we run \Cref{alg:proposed} for $K$ iterations with parameters $\{\bm{T}_i\}^{K}_{i=1}$ and $\{\bm{T}'_i\}^{K}_{i=1}$ whose difference lies only in the $k$th element. Then, we have the following with probability at least $1-\varepsilon$:
\begin{align*}
    \left\|\bm{P}^{(K+1)}-\bm{P}'^{(K+1)}\right\|_F \leq 2 \mu\eta U \qty{1 + \qty(K-k + \frac{\zeta}{\eta})N}.
\end{align*}
\end{lemma}
Now, we are ready to show that \Cref{alg:proposed} is PAC uniformly stable.
\begin{theorem}\label{prop:stability}
Assume that $\alpha^{(k)} \in (0,1]$ for all $k\in [K]$.
Then, \Cref{alg:proposed} is PAC uniformly stable.
\end{theorem}

From \Cref{th:general,prop:stability}, we obtain an upper bound of the generalization gap of \Cref{alg:proposed} if \Cref{alg:proposed} is decomposable.

\section{Numerical Experiments}\label{sec:numer}

To verify that the Pareto set can be accurately approximated by a B\'ezier simplex obtained by our proposed method~(\Cref{alg:proposed}), we applied the proposed method to three multi-objective problems, which are known to be simplicial.
Notice that the formulation of skew-$M$MMD includes some important problems, such as the group Lasso in sparse modeling \cite{Yuan2006}.
The definition of each problem instance is shown in \Cref{appendix:problem}.
In \Cref{alg:proposed}, we set the degree of B\'ezier simplex as $D=3$, the initial control points as $\bm{P}^{(1)}=\bm O$, which is the zero matrix. 
Also, we set $K=1000$ and $\alpha^{(k)}=\frac{1}{k}$. 
The number of points to be sampled from a simplex in each iteration was tested for $N\in \{30, 50, 100\}$.

As a baseline, we used NSGA-II with the B\'ezier simplex fitting~\citep{Kobayashi2019}. 
Specifically, we obtain approximated Pareto solution samples by NSGA-II \cite{deb2000fast} implemented in jMetal \cite{benitez2019jmetalpy}, which is provided under MIT license, with default parameters except for population size. 
Then, we fit a B\'ezier simplex of degree $D=3$ to the approximated Pareto solutions by the all-at-once method proposed in \citet{Kobayashi2019}. 
We set the number of population size as $p\in \{30, 50,100\}$.
We implemented these algorithms in Python 3.8.10, and the experiments were performed on a Windows 10 PC with an Intel(R) Xeon(R) W-1270 CPU \SI{3.40}{GHz} and \SI{64}{GB} RAM.

\subsection{MSEs comparison}
First, we picked up a simplicial problem instance whose $\bm x^\star$ is analytically obtained and evaluated how accurate an obtained B\'ezier simplex approximates the Pareto set.
In this experiment, we used scaled-MED, which is the three-objective problem with three variables and is known to be simplicial.
The problem definition is shown in Appendix~\ref{appendix:problem}.
To evaluate the approximation accuracy of the estimated B\'ezier simplex, we used the mean squared error (MSE) defined by $\mathrm{MSE} \coloneqq \frac{1}{N} \sum^{N}_{n=1} \| \bm{b}(\hat{\bm{t}}_n|\bm{P}) - \bm{x}^\star(\hat{\bm{t}}_n)\|^2_2$, where $\bm{x}^\star$ is a map from a weight to the minimizer of the corresponding scalarizing function. The map $\bm{x}^\star$ for scaled-MED is shown in Appendix \ref{appendix:scaledmed_opt}.
To calculate MSE, we randomly sample  $\{\hat{\bm{t}}_n\}^{10000}_{n=1}$ i.i.d. from the uniform distribution on $\Delta^{2}$.
We repeated the experiments 20 times with different parameters and computed the average and the standard deviations of MSEs.

Table \ref{table:MSE} shows the average and the standard deviation of the MSEs with $N\in\{30,50,100\}$ for \Cref{alg:proposed} and $p\in\{30,50,100\}$ for NSGA-II. 
In Table \ref{table:MSE}, we highlighted the best score of MSE out of the proposed and baseline method where the difference is significant with the significance level $\mathrm{p}=0.01$ by the Wilcoxon rank-sum test. 
\Cref{table:MSE} shows that the B\'ezier simplex obtained by our proposed method can represent Pareto set well. Also, the MSEs of our method decrease with larger $N$. This result supports the PAC uniform stability of \Cref{alg:proposed}.
\Cref{fig:proposed_bezier} shows the B\'ezier simplex obtained by our proposed method, and \Cref{fig:existing_bezier} shows the B\'ezier simplex obtained by the all-at-once with NSGA-II. The true Pareto set of scaled-MED for our setting is known to be a curved triangle that can be triangulated into three vertices. In \Cref{fig:proposed_bezier}, the B\'ezier simplex obtained by our method approximates the Pareto set well even with $N=30$.

\begin{table*}[ht]
\renewcommand{\arraystretch}{0.8}
\caption{MSE (avg.$\pm$s.d. over 20 trials) for scaled-MED.}
\label{table:MSE}
\centering
\begin{tabular}{ccc|cccccccc}
\toprule
Problem & \multicolumn{2}{c|}{Proposed} & \multicolumn{2}{c}{NSGA-II + all-at-once} \\
\midrule
scaled-MED & $N=30$ & \textbf{5.51e-05$\pm$2.40e-06} & $p=30$ & 1.51e-01$\pm$1.56e-03 \\
 & $N=50$ & \textbf{4.43e-05$\pm$1.88e-06} & $p=50$ & 8.23e-02$\pm$6.80e-04 \\
 & $N=100$ & \textbf{3.78e-05$\pm$1.30e-06} & $p=100$ & 1.25e-01$\pm$1.32e-03 \\
\bottomrule
\end{tabular}
\end{table*}
\begin{table*}[ht]
\caption{GD and IGD (avg.$\pm$s.d. over 20 trials) for skew-3MED and skew-3MMD.}
\renewcommand{\arraystretch}{0.8}
\label{table:GDIGD}
\centering
\begin{tabular}{cccc|cccccc}
\toprule
Problem &  & \multicolumn{2}{c|}{Proposed} & \multicolumn{2}{c}{NSGA-II + all-at-once} \\
\midrule
skew-3MED & GD & $N=30$ & \textbf{6.22e-02$\pm$1.46e-03} & $p=30$ & 1.76e-01$\pm$3.67e-03 \\
 & & $N=50$ & \textbf{5.50e-02$\pm$4.57e-04} & $p=50$ & 1.22e-01$\pm$2.00e-03 \\
 & & $N=100$ & \textbf{5.38e-02$\pm$7.90e-04} & $p=100$ & 6.33e-02$\pm$7.90e-04 \\
 & IGD & $N=30$ & \textbf{9.50e-02$\pm$3.68e-04} & $p=30$ & 1.38e-01$\pm$8.63e-04 \\
 & & $N=50$ & \textbf{8.84e-02$\pm$3.29e-04} & $p=50$ & 1.33e-01$\pm$8.92e-04 \\
 & & $N=100$ & \textbf{8.45e-02$\pm$3.42e-04} & $p=100$ & 8.86e-02$\pm$6.02e-04 \\
\midrule 
skew-3MMD & GD & $N=30$ & \textbf{5.45e-02$\pm$1.32e-03} & $p=30$ & 2.04e-01$\pm$6.07e-03 \\
 & & $N=50$ & \textbf{5.16e-02$\pm$9.85e-04} & $p=50$ & 9.34e-01$\pm$2.15e-03 \\
 & & $N=100$ & \textbf{5.06e-02$\pm$1.07e-03}  & $p=100$ & 6.60e-02$\pm$1.07e-03 \\
 & IGD & $N=30$ & \textbf{6.40e-02$\pm$4.69e-04} & $p=30$ & 1.00e-01$\pm$1.45e-03 \\
 & & $N=50$ & \textbf{6.38e-02$\pm$4.27e-04} & $p=50$ & 7.64e-02$\pm$9.55e-04 \\
 & & $N=100$ & \textbf{6.54e-02$\pm$4.81e-04} & $p=100$ & 7.23e-02$\pm$8.46e-04 \\
\bottomrule
\end{tabular}
\end{table*}

\subsection{GDs and IGDs comparison}\label{subsec:gd}
Next, we validate the practicality of the proposed method in more practical settings. 
In this experiment, we used two simplicial problem instances: skew-$3$MED and skew-$3$MMD, whose $\bm x^\star$  cannot be represented in a closed-form.
We show their definitions in Appendix \ref{appendix:problem}. 
We used the generational distance (GD) \cite{Veldhuizen99} and the inverted generational distance (IGD) \cite{Zitzler03} to evaluate how accurately the estimated B\'ezier simplex approximates the Pareto set.
GD and IGD are defined by:
\begin{align*}
    \mathrm{GD}(X,Y)\coloneqq\frac{1}{|X|}\sum_{\bm{x}\in X}\min_{\bm{y}\in Y} \|\bm{x}-\bm{y}\|_2, \quad\mathrm{IGD}(X,Y)\coloneqq\frac{1}{|Y|}\sum_{\bm{y
    }\in Y}\min_{\bm{x}\in X} \|\bm{x}-\bm{y}\|_2,
\end{align*}
where $X$ is a finite set whose elements are sampled from an estimated hyper-surface and $Y$ is a validation set. 
We can say that the obtained B\'ezier simplex is close to the Pareto set if and only if both GD and IGD are small.
As a validation set $Y$, we generated approximate Pareto solutions by NSGA-II with the population size of $1000$.  
To construct $X$, we randomly sample $\{\hat{\bm{t}}_n\}^{1000}_{n=1}$ i.i.d. from the uniform distribution on $\Delta^2$ and obtain sample points on the estimated B\'ezier simplex. 
We repeated the experiments 20 times with different parameters and computed the average and the standard deviations of their GDs and IGDs.

Table \ref{table:GDIGD} shows the average and the standard deviation of the GDs and IGDs when $N\in\{30, 50, 100\}$ and $p\in\{30,50,100\}$. 
In Table \ref{table:GDIGD}, we highlighted the best score of GD and IGD  where
the difference is at a significant with significance level $\mathrm{p} = 0.01$ by the Wilcoxon rank-sum test.
Table \ref{table:GDIGD} shows that the proposed method achieved better GD and IGD for skew-3MED and skew-3MMD.
The differences are pronounced in the results of small sample/population size, which implies our method obtains a B\'ezier simplex approximating Pareto set well.
\begin{figure*}[t]
    \centering
    \subfigure[$N=30$]{\includegraphics[width=0.3\linewidth]{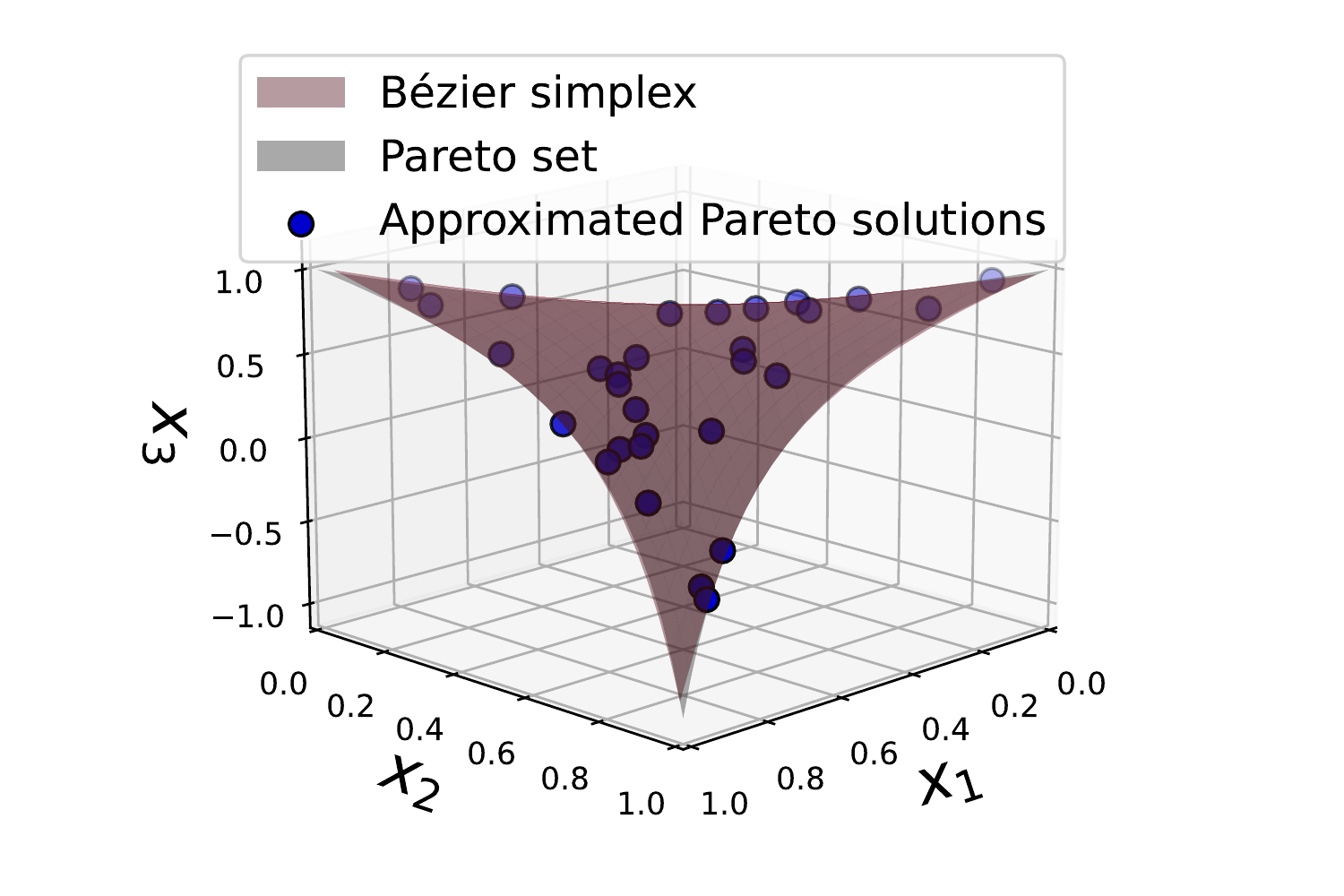}\label{fig:proposed_n30}}%
    \subfigure[$N=50$]{\includegraphics[ width=0.3\linewidth]{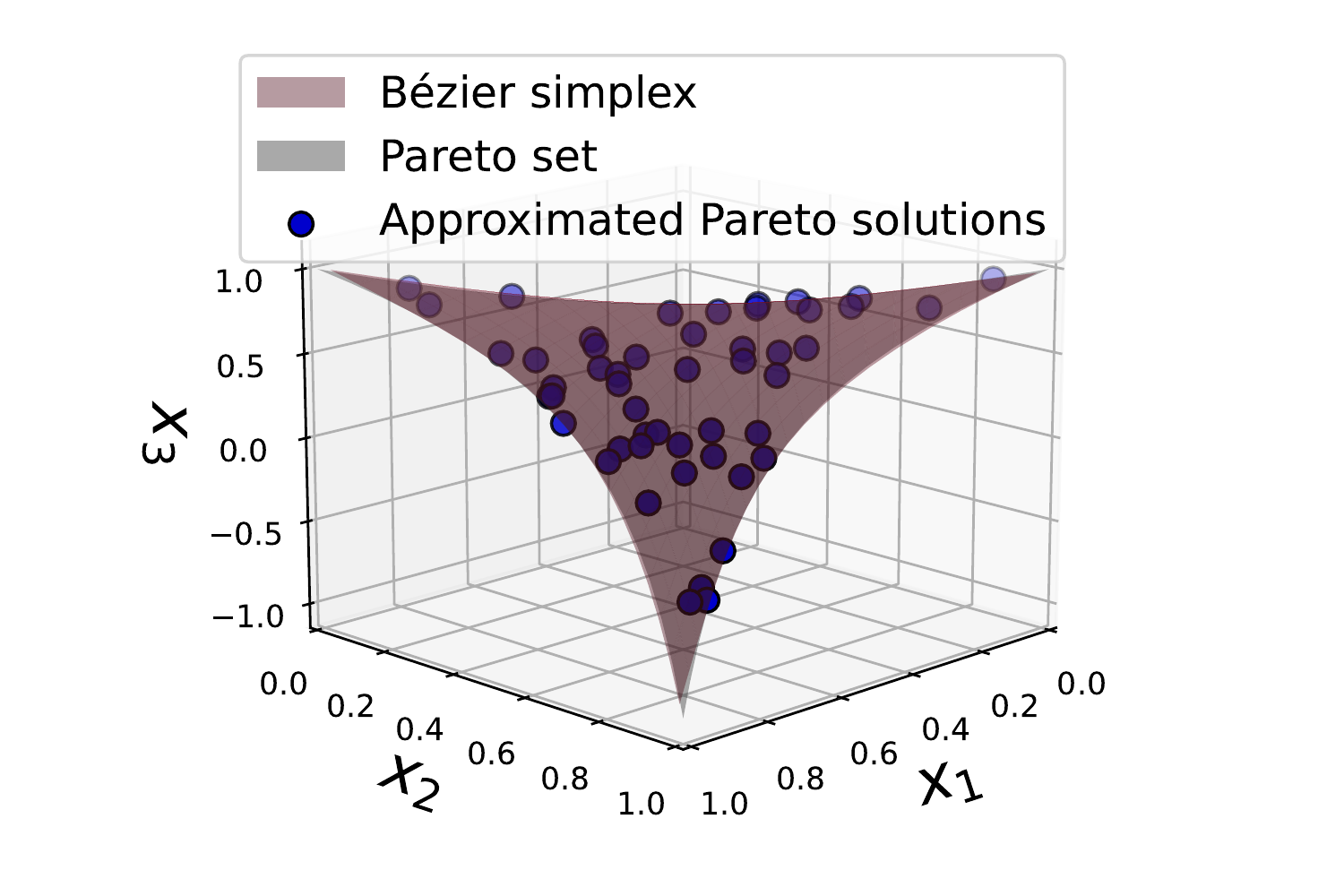}\label{fig:proposed_n50}}%
    \subfigure[$N=100$]{\includegraphics[ width=0.3\linewidth]{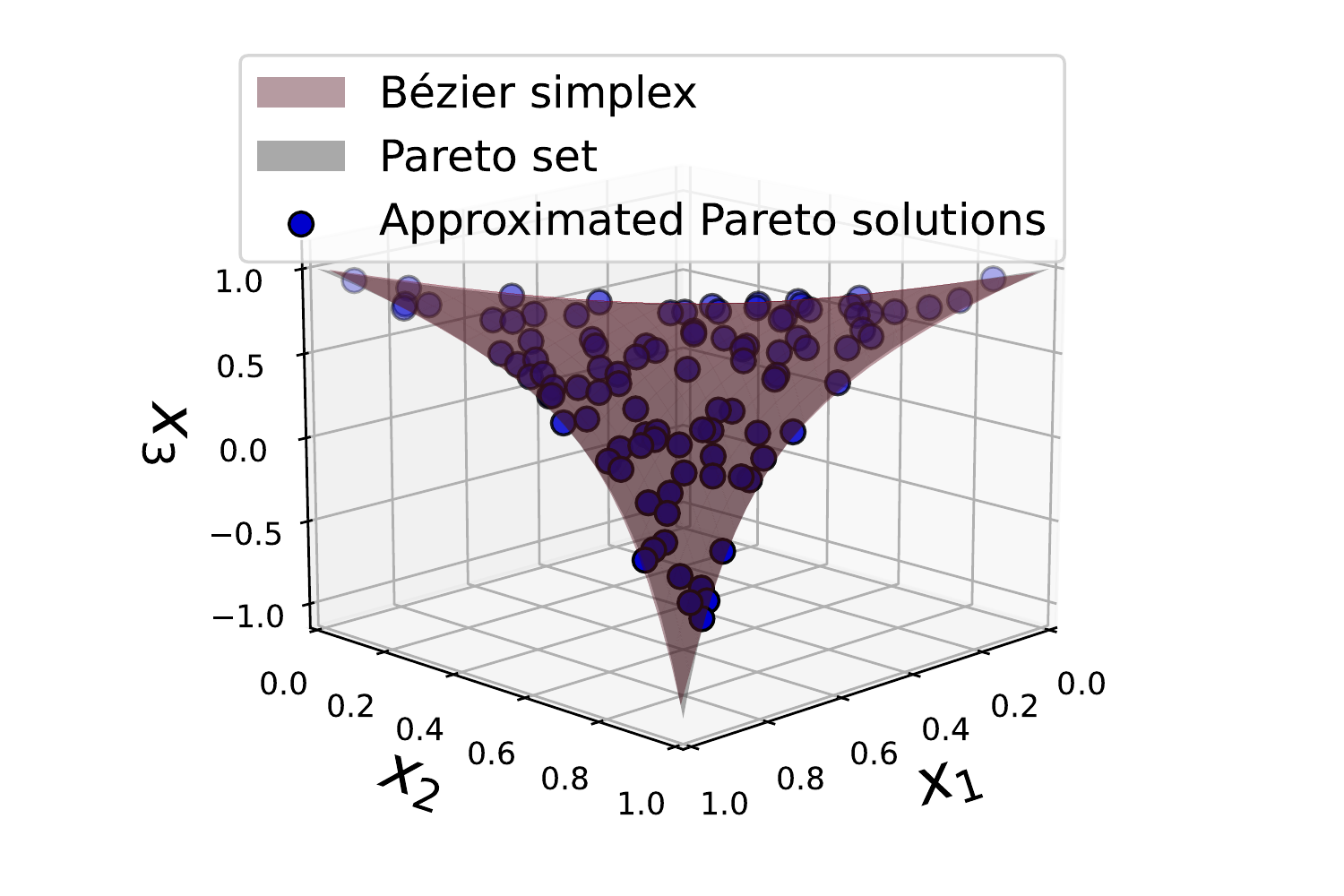}\label{fig:proposed_n100}}
    \caption{Results for \Cref{alg:proposed} with the sample size of $30,50$, and $100$.}
    \label{fig:proposed_bezier}
    \subfigure[$p=30$]{\includegraphics[width=0.3\linewidth]{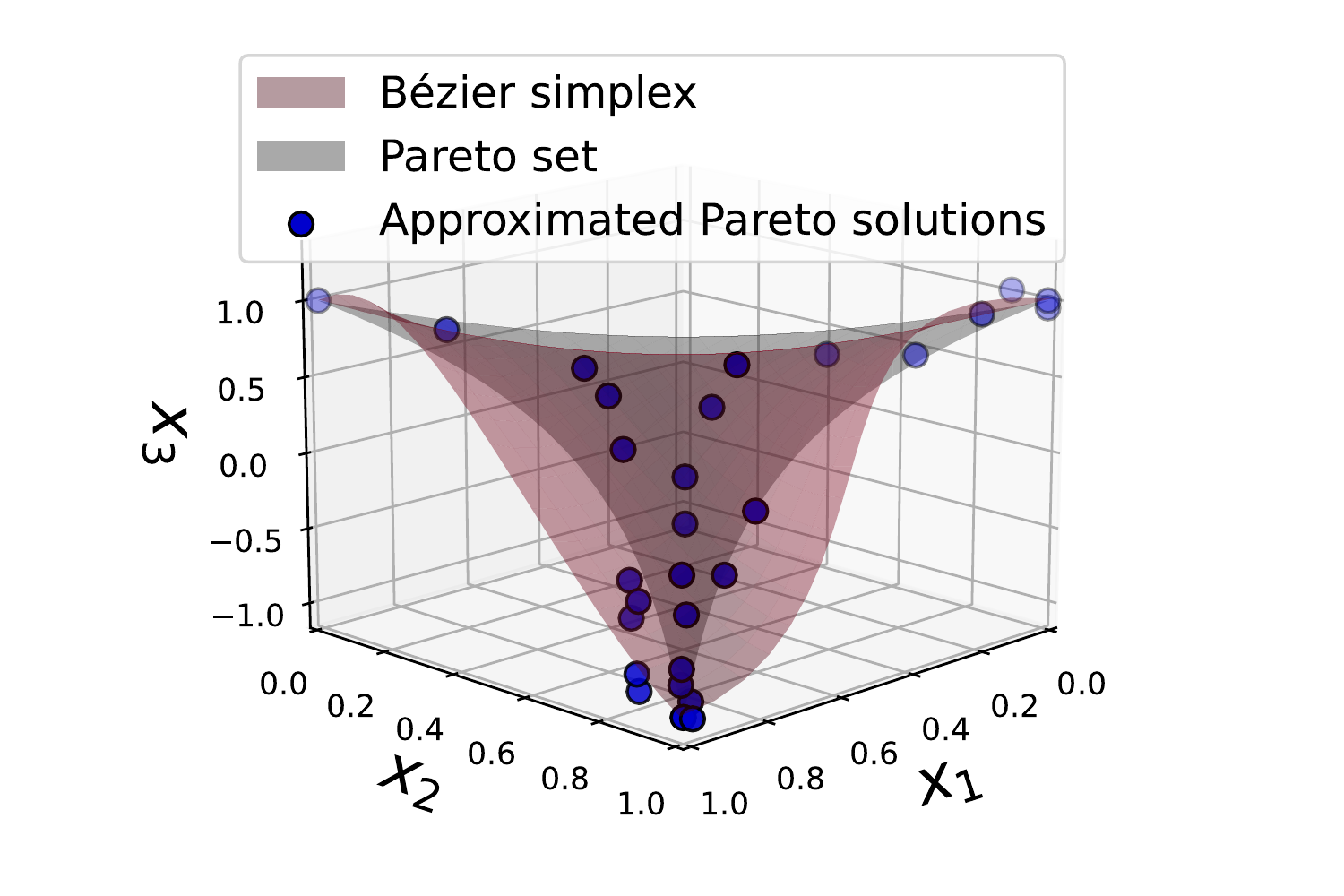}\label{fig:jmetal_n30}}%
    \subfigure[$p=50$]{\includegraphics[ width=0.3\linewidth]{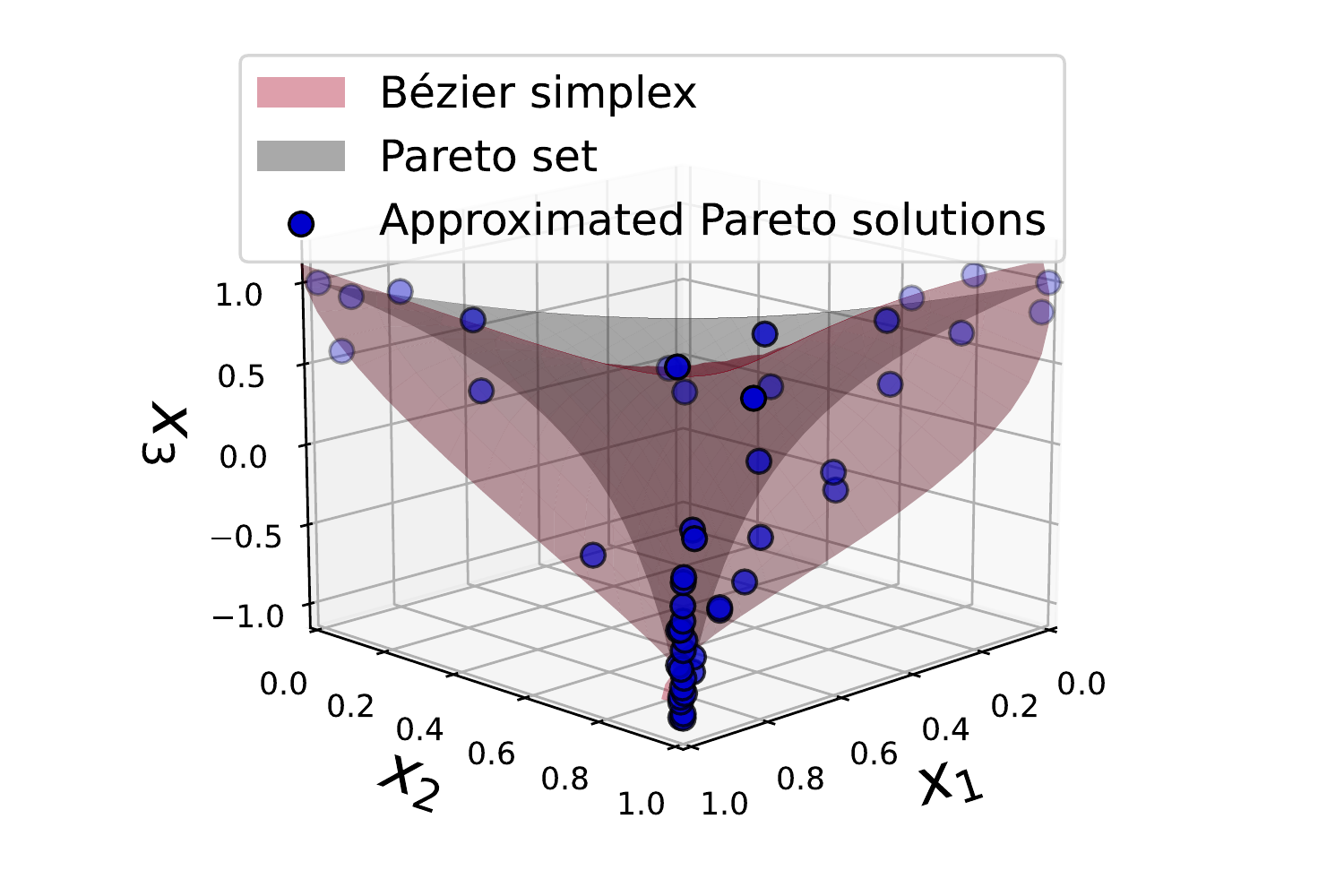}\label{fig:jmetal_n50}}
    \subfigure[$p=100$]{\includegraphics[ width=0.3\linewidth]{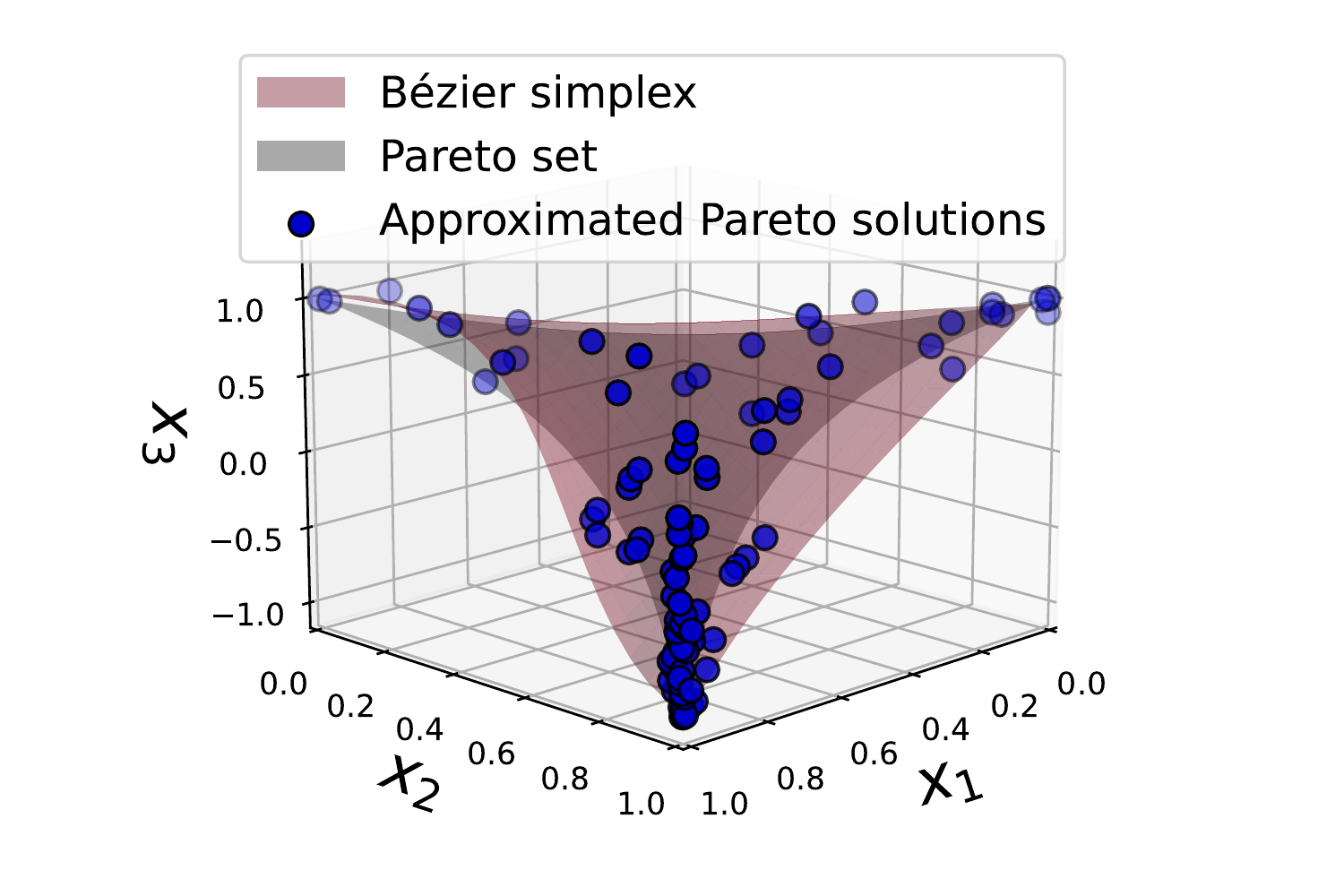}\label{fig:jmetal_n100}}
    \caption{Results for NSGA-II and the all-at-once with the population size of $30,50$, and $100$.}
    \label{fig:existing_bezier}
\end{figure*}

\section{Conclusion}
In this paper, we have devised a general strategy to construct a multi-objective optimization algorithm from a single objective method with the B\'ezier simplex.
Also, we have defined the PAC stability of optimization algorithms and proved that this stability gives us an upper bound on the generalization gap in the sense of PAC learning.
Our theoretical analysis showed that if we construct a multi-objective optimization algorithm from a gradient-based single-objective optimization algorithm, the resultant algorithm is PAC stable.
In our numerical experiments, we have demonstrated the multi-objective optimization algorithm on the basis of our scheme gives better generalization gaps and approximation accuracies of the Pareto set than the existing algorithm for simplicial problems. 
As a concluding remark, we have to note that this study is limited to treat simplicial problems.
It would be interesting for future studies to extend this study to non-simplicial cases.

\bibliographystyle{plainnat}
\bibliography{cite}

\newpage
\appendix

\section{Proof of \Cref{th:general}}
\begin{proof}
We denote two independent random samples by $S=\left(\bm{t}_{1}, \ldots, \bm{t}_{N}\right)$, $S^{\prime}=\left(\bm{t}_{1}^{\prime}, \ldots, \bm{t}_{N}^{\prime}\right)$. Let $S^{(i)}=\left(\bm{t}_{1}, \ldots, \bm{t}_{i-1}, \bm{t}_{i}^{\prime}, \bm{t}_{i+1}, \dots, \bm{t}_{N}\right)$ be the sample that is same as $S$ except in the $i$th example where we replace $t_{i}$ with $t_{i}^{\prime}$.
Since $D_{\varepsilon}=B_{\varepsilon}^{N+1}$, $\left(\bm{t}_{1}, \ldots, \bm{t}_{i}, \bm{t}_{i}^{\prime}, \bm{t}_{i+1}, \dots, \bm{t}_{N}\right) \in D_{\varepsilon}$ for any $i$ if and only if $S, S^{\prime}\in C_{\varepsilon}:=B_{\varepsilon}^N$.
In this case, we have
\begin{align*}
 \mathbb{E}_A\left[ \left| \ell(A(S);\bm{t}_i)-\ell(A(S^{(i)});\bm{t}_i) \right| \right] < \delta.
\end{align*}
Then adding the inequalities for $i$ and applying the triangle inequality, we obtain
\begin{align*}
\left| \mathbb{E}_A\left[  \frac{1}{N}\sum_{i=1}^N \ell(A(S);\bm{t}_i)-\frac{1}{N}\sum_{i=1}^N \ell(A(S^{(i)});\bm{t}_i) \right] \right| < \delta.
\end{align*}Let us denote the conditional probability distribution of $\mathcal{D}, \mathcal{D}^N$ under the condition $B_{\varepsilon}, C_{\varepsilon}$ by $\mathcal{B}_{\varepsilon}, \mathcal{C}_{\varepsilon}$ respectively. Then we have

\begin{align*}
\left| \mathbb{E}_{(S, S') \sim \mathcal{C}_{\varepsilon}^2} \mathbb{E}_A\left[  \frac{1}{N}\sum_{i=1}^N \ell(A(S);\bm{t}_i)-\frac{1}{N}\sum_{i=1}^N \ell(A(S^{(i)});\bm{t}_i) \right] \right| < \delta.
\end{align*}
Here, we have
\begin{align*}
 &\mathbb{E}_{(S, S') \sim \mathcal{C}_{\varepsilon}^2} \mathbb{E}_{A}\left[\frac{1}{N} \sum_{i=1}^{N} \ell \left(A(S) ; \bm{t}_{i}\right)\right]  = \mathbb{E}_{S\sim \mathcal{C}_{\varepsilon}} \mathbb{E}_{A}\left[\frac{1}{N} \sum_{i=1}^{N} \ell \left(A(S) ; \bm{t}_{i}\right)\right]  = \hat{\mathbb{E}}_{S} \mathbb{E}_{A}\left[R_{S}[A(S)]\right],
\end{align*}
and
\begin{align*}
 \mathbb{E}_{(S, S') \sim \mathcal{C}_{\varepsilon}^2} \mathbb{E}_{A}\left[\frac{1}{N} \sum_{i=1}^{N} \ell \left(A(S^{(i)}) ; \bm{t}_{i}\right)\right] &=  \mathbb{E}_{(S, S') \sim \mathcal{C}_{\varepsilon}^2}   \mathbb{E}_{A}\left[\frac{1}{N} \sum_{i=1}^{N} \ell \left(A(S) ; \bm{t}'_{i}\right)\right]  \\
&=\mathbb{E}_{S \sim \mathcal{C}_{\varepsilon}}  \mathbb{E}_{A}\mathbb{E}_{S' \sim \mathcal{C}_{\varepsilon}} \left[\frac{1}{N} \sum_{i=1}^{N} \ell \left(A(S) ; \bm{t}'_{i}\right)\right]\\
 &= \hat{\mathbb{E}}_{S} \mathbb{E}_{A}\left[\hat{R}[A(S)]\right],
\end{align*}
where $\hat{\mathbb{E}}_{S}$ is the conditional expectation value of $C_{\varepsilon}$ and $\hat{R}[A(S)]$ is the conditional generalization error of $C_{\varepsilon}$. Thus we obtain the inequality in the theorem.

Finally, we have
\begin{align*}
    \mathbb{P}\qty( C_\varepsilon) = \mathbb{P}\qty( B_\varepsilon)^N =\mathbb{P}\qty( D_\varepsilon)^{\frac{N}{N+1}} > \mathbb{P}\qty( D_\varepsilon) > 1-\varepsilon,
\end{align*}
which completes the proof.
\end{proof}
\section{Proofs of \Cref{lemma:sigma-bounded,lemma:growth}}\label{proof:lemma:lammin}
We first show the three lemmas in advance.
\begin{lemma}\label{lemma:lammin}
For all $\varepsilon\in (0,1)$, there exists $\eta > 0$ satisfying
\begin{align*}
    \mathbb{P}\qty( \min_i \lambda_{\min}\qty(\bm{Z}(\bm{T}_i)^\top\bm{Z}(\bm{T}_i)) > \eta) \geq 1 - \varepsilon,
\end{align*}
where $\lambda_{\min}(\bm{A})$ denotes the minimal eigenvalue of $\bm{A}$ and $\bm{T}_i=\{\bm{t}^{(i)}_n\}^N_{n=1}$ is drawn i.i.d. from the uniform distribution on $\Delta^{M-1}$ for $i\in[K+1]$.
\end{lemma}
\begin{proof}
Consider the set
\begin{align*}
    A_0 \coloneqq
    \Set*{\{\bm{T_i}\}_{i=1}^{K+1}\subseteq(\Delta^{M-1})^{K+1}}{\min_i \lambda_{\min}\qty(\bm{Z}(\bm{T_i})^\top\bm{Z}(\bm{T_i}))=0}.
\end{align*}
Since $\bm{Z}(\bm{T}_i)^\top\bm{Z}(\bm{T}_i)$ is a symmetric matrix, it has non-negative eigenvalues and $\bm{Z}(\bm{T}_i)^\top\bm{Z}(\bm{T}_i)$  has zero eigenvalue if and only if the determinant of $\bm{Z}(\bm{T}_i)^\top\bm{Z}(\bm{T}_i)$ is zero. Hence, $A_0$ is a subset of 
\begin{align*}
    \Set*{\{\bm{t}_n\}_{n=1}^{N}\subseteq\Delta^{M-1}}{\prod_i \det(\bm{Z}(\bm{T}_i)^\top\bm{Z}(\bm{T}_i) )=0}.
\end{align*}
Therefore, $A_0$ is equal to the zero set of a polynomial.  This implies  $\mathbb{P}(A_0)=0$.  Considering the set
\begin{align*}
    A_{\eta} \coloneqq
    \Set*{\{\bm{T}_i\}_{i=1}^{K+1}\subseteq(\Delta^{M-1})^{K+1}}{\min_i \lambda_{\min}\qty(\bm{Z}(\bm{T_i})^\top\bm{Z}(\bm{T_i}))\leq\eta},
\end{align*}
then we have
\begin{align*}
    \mathbb{P}(A_\eta) \to \mathbb{P}(A_0) = 0 \quad (\eta\to 0).
\end{align*}
This implies that for all $\varepsilon\in (0,1)$ there exists some $\eta > 0$ such that
\begin{align}
    \Pr(\Set*{\{\bm{T}_i\}_{i=1}^{K+1}\subseteq(\Delta^{M-1})^{K+1}}{\min_i \lambda_{\min}\qty(\bm{Z}(\bm{T}_i)^\top\bm{Z}(\bm{T}_i))\leq\eta}) < \varepsilon.
\end{align}
The proof is completed by taking the complementary event.
\end{proof}
\begin{lemma}\label{col:ZZ-bounded}
For all $\varepsilon\in (0,1)$, there exists $\zeta > 0$ satisfying
\begin{align}
    \mathbb{P}\qty(\max_{i}\left\|\bm{Z}(\bm{T}_i)^\top\bm{Z}(\bm{T}_i)\right\|_F < \zeta) \geq 1 - \varepsilon,
\end{align}
where $\bm{T}_i=\{\bm{t}^{(i)}_n\}_{n=1}^{N}$ is drawn i.i.d. from the uniform distribution on $\Delta^{M-1}$ for $i\in [K+1]$.
\end{lemma}
\begin{proof}
Since $\bm{Z}(\bm{T}_i)^\top\bm{Z}(\bm{T}_i)$ is a symmetric matrix, there exists an orthogonal matrix $\bm{Q}_i$ such that $\bm{Z}(\bm{T}_i)^\top\bm{Z}(\bm{T}_i)=\bm{Q}_i\bm{\Lambda}_i\bm{Q}_i^\top$ where $\bm{\Lambda}_i$ is a diagonal matrix whose diagonal entry is the eigenvalue of $\bm{Z}(\bm{T}_i)^\top\bm{Z}(\bm{T}_i)$. According to Lemma \ref{lemma:lammin}, $\bm{Z}(\bm{T}_i)^\top\bm{Z}(\bm{T}_i)$ is a regular matrix with probability at least $1-\varepsilon$ for all $i\in [K+1]$. Hence, we have the following with probability at least $1-\varepsilon$:
\begin{align*}
    \max_{i}\left\|\qty(\bm{Z}(\bm{T}_i)^\top\bm{Z}(\bm{T}_i))^{-1}\right\|_F
    &= \max_{i}\|\bm{Q}_i^\top\bm{\Lambda}_i^{-1}\bm{Q}_i\|_F \\
    &= \max_{i}\|\bm{\Lambda}_i^{-1}\|_F \\
    &= \max_{i}\sqrt{\sum_{n=1}^{|\mathbb{N}^M_D|} \frac{1}{\lambda^2_n(\bm{Z}(\bm{T}_i)^\top\bm{Z}(\bm{T}_i))}} \\
    &\leq \sqrt{ \frac{|\mathbb{N}^M_D|}{\min_{i}\lambda^2_{\min}(\bm{Z}(\bm{T}_i)^\top\bm{Z}(\bm{T}_i))}} \\
    &< \frac{\sqrt{|\mathbb{N}^M_D|}}{\eta},
\end{align*}
where the second equality follows from the fact that the Frobenius norm is unitarily invariant, and the second inequality follows from \Cref{lemma:lammin}. The proof is completed by letting $\zeta$ be some real number greater than or equal to $\frac{\sqrt{|\mathbb{N}^M_D|}}{\eta}$.
\end{proof}
\begin{lemma}\label{lemma:ZG}
Let $U>0$ be a constant satisfying $\max_{\bm{t}\in\Delta^{M-1}}\|\bm{z}(\bm{t})\|_2 \leq U$.
Then, for any $\bm{T}\subseteq\Delta^{M-1}$, we have
\begin{align}
    \left\|\bm{Z}(\bm{T})^\top\bm{G}(\bm{T})\right\|_F \leq NU\mu.
\end{align}
\end{lemma}
\begin{proof}
First, we show that $\|\bm{z}(\bm{t})\|_2$ is bounded above for any $\bm{t}\in\Delta^{M-1}$. Since $\bm{z}$ is a continuous function over $\bm{t}$ whose domain $\Delta^{M-1}$ is compact, there exists upper bound $U>0$ for any $\bm{t}\in\Delta^{M-1}$. 
Next, we show that $\|\bm{Z}(\bm{T})^\top\bm{G}(\bm{T})\|_F$ is bounded above.
For any $\bm{T}\coloneqq\qty{\bm{t}_n}^{N}_{n=1}\subseteq\Delta^{M-1}$, we have
\begin{align*}
    \bm{Z}(\bm{T})^\top\bm{G}(\bm{T})
    = (\bm{z}_1,\dots,\bm{z}_N)
    \begin{bmatrix}
   \bm{t}_1^\top J_{\bm{f}}(\bm{P}^\top \bm{z}_1)\\
   \vdots\\
   \bm{t}_N^\top J_{\bm{f}}(\bm{P}^\top \bm{z}_N)\\
   \end{bmatrix} 
   = \sum_{n=1}^{N} \bm{z}_n \bm{t}_n^\top J_{\bm{f}}(\bm{P}^\top\bm{z}_n).
\end{align*}
Therefore,
\begin{align*}
    \left\|\bm{Z}(\bm{T})^\top\bm{G}(\bm{T})\right\|_F 
    &= \left\|\sum_{n=1}^{N} \bm{z}_n \bm{t}_n^\top J_{\bm{f}}(\bm{P}^\top\bm{z}_n)\right\|_F \\
    &\leq \sum_{n=1}^{N} \left\|\bm{z}_n \bm{t}_n^\top J_{\bm{f}}(\bm{P}^\top\bm{z}_n)\right\|_F \\
    &\leq \sum_{n=1}^{N} \left\|\bm{z}_n \right\|_2 \left\| \bm{t}_n^\top J_{\bm{f}}(\bm{P}^\top\bm{z}_n)\right\|_2 \\
    &=\sum_{n=1}^{N} \left\|\bm{z}_n \right\|_2 \left\| \sum_{m=1}^{M}t_{nm}\nabla f_m (\bm{P}^\top\bm{z})\right\|_2 \\
    &\leq \sum_{n=1}^{N} U\mu = NU\mu.
\end{align*}
The last inequality holds by the fact that the term $\sum_{m=1}^{M}t_{nm}\nabla f_m (\bm{P}^\top\bm{z})$ is a convex combination of $\nabla f_1 (\bm{P}^\top\bm{z}),\dots,\nabla f_M (\bm{P}^\top\bm{z})$ and the assumption that every function $f_1,\dots,f_M$ is $\mu$-Lipschitz continuous.
\end{proof}

Finally, we show \Cref{lemma:sigma-bounded,lemma:growth}.
\begin{proof}[Proof of \Cref{lemma:sigma-bounded}]
By \eqref{eq:update_p}, we have
\begin{align*}
    \left\| \varphi_{\bm{T}}(\bm{P}) - \bm{P} \right\|_F 
    &= \alpha^{(k)} \left\| \qty(\bm{Z}^\top\bm{Z})^{-1}\bm{Z}^\top\bm{G} \right\|_F \\ 
    &\leq  \left\| \qty(\bm{Z}^\top\bm{Z})^{-1} \right\|_F \left\|\bm{Z}^\top\bm{G} \right\|_F.
\end{align*}
Let $\eta>0$ be a constant as in \Cref{col:ZZ-bounded}.
From \Cref{col:ZZ-bounded,lemma:ZG}, we have the following with probability at least $1-\varepsilon$:
\begin{align*}
    \left\| \varphi_{\bm{t}}(\bm{P}) - \bm{P} \right\|_F \leq \eta NU\mu.
\end{align*}
\end{proof}
\begin{proof}[Proof of \Cref{lemma:growth}] 
Let $\bm{T}$, $\bm{T}'$ and $\tilde{\bm{T}}$ be
\begin{align*}
    \bm{T} &= \qty{\bm{t}_1,\dots,\bm{t}_{N-1},\bm{t}_N},\\
    \bm{T}' &= \qty{\bm{t}_1,\dots,\bm{t}_{N-1},\bm{t}'_N},\\
    \tilde{\bm{T}} &= \qty{\bm{t}_1,\dots,\bm{t}_{N-1},\bm{t}_N,\bm{t}'_N}.
\end{align*}
Let $\tilde{\bm{Z}}$ be a matrix constructed by $\tilde{\bm{T}}$. By Sherman-Morrison formula, we have
\begin{align*}
    (\tilde{\bm{Z}}^\top\tilde{\bm{Z}})^{-1}
    &= (\bm{Z}^\top\bm{Z} + \bm{z}_{N+1}\bm{z}_{N+1}^\top)^{-1} \\
    &= (\bm{Z}^\top\bm{Z})^{-1} + \frac{(\bm{Z}^\top\bm{Z})^{-1}\bm{z}_{N+1}\bm{z}_{N+1}^\top(\bm{Z}^\top\bm{Z})^{-1}}{1+\bm{z}_{N+1}^\top(\bm{Z}^\top\bm{Z})^{-1}\bm{z}_{N+1}}.
\end{align*}
Let $\varphi_{\bm{T}}(\bm{P})$ be the control points obtained by \Cref{alg:proposed} with $\bm{T}$. Then, we have
\begin{align*}
    \varphi_{\tilde{\bm{T}}}(\bm{P}) - \varphi_{\bm{T}}(\bm{P}) &= (\tilde{\bm{Z}}^\top\tilde{\bm{Z}})^{-1}\tilde{\bm{Z}}^\top\tilde{\bm{G}} - (\bm{Z}^\top\bm{Z})^{-1} \bm{Z}^\top\bm{G} \\
    &= (\bm{Z}^\top\bm{Z})^{-1}(\tilde{\bm{Z}}^\top\tilde{\bm{G}}-\bm{Z}^\top\bm{G}) + \frac{(\bm{Z}^\top\bm{Z})^{-1}\bm{z}_{N+1}\bm{z}_{N+1}^\top(\bm{Z}^\top\bm{Z})^{-1}\tilde{\bm{Z}}^\top\tilde{\bm{G}}}{1+\bm{z}_{N+1}^\top(\bm{Z}^\top\bm{Z})^{-1}\bm{z}_{N+1}} \\
    &= (\bm{Z}^\top\bm{Z})^{-1}(\bm{z}_{N+1}\bm{t}_{N+1}^\top J_{\bm{f}}(\bm{P}^\top\bm{z}_{N+1})) + \frac{(\bm{Z}^\top\bm{Z})^{-1}\bm{z}_{N+1}\bm{z}_{N+1}^\top(\bm{Z}^\top\bm{Z})^{-1}\tilde{\bm{Z}}^\top\tilde{\bm{G}}}{1+\bm{z}_{N+1}^\top(\bm{Z}^\top\bm{Z})^{-1}\bm{z}_{N+1}}.
\end{align*}
Considering the norm on both sides, we have
\begin{align*}
    \|\varphi_{\tilde{\bm{T}}}(\bm{P}) - \varphi_{\bm{T}}(\bm{P})\|_F 
    &\leq \|(\bm{Z}^\top\bm{Z})^{-1}\|_F\cdot \|\bm{z}_{N+1}\|_2\cdot\|\bm{t}_{N+1}^\top J_{\bm{f}}(\bm{P}^\top\bm{z}_{N+1})\|_2 \\
    &\quad+ \|\tilde{\bm{Z}}^\top\tilde{\bm{G}}\|_F\cdot\left\|\frac{(\bm{Z}^\top\bm{Z})^{-1}\bm{z}_{N+1}\bm{z}_{N+1}^\top(\bm{Z}^\top\bm{Z})^{-1}}{1+\bm{z}_{N+1}^\top(\bm{Z}^\top\bm{Z})^{-1}\bm{z}_{N+1}}\right\|_F
\end{align*}
In the following, for the sake of simplicity, let $\bm{A}=\bm{Z}^\top\bm{Z},\bm{b}=\bm{z}_{N+1}$ and $\bm{y}=\bm{A}^{-1}\bm{b}$. Then, we have the following inequality with probability at least $1-\varepsilon$:
\begin{align*}
    \left\|\frac{(\bm{Z}^\top\bm{Z})^{-1}\bm{z}_{N+1}\bm{z}_{N+1}^\top(\bm{Z}^\top\bm{Z})^{-1}}{1+\bm{z}_{N+1}^\top(\bm{Z}^\top\bm{Z})^{-1}\bm{z}_{N+1}}\right\|_F
    &= \left\|\frac{\bm{A}^{-1}\bm{b}\bm{b}^\top\bm{A}^{-1}}{1+\bm{b}^\top\bm{A}^{-1}\bm{b}}\right\|_F \\
    &= \left\|\frac{(\bm{b}^\top\bm{A}^{-1})^\top (\bm{b}^\top\bm{A}^{-1})}{1+\bm{b}^\top\bm{A}^{-1}(\bm{A}\bm{A}^{-1})\bm{b}}\right\|_F \\
    &= \left\| \frac{\bm{y}\bm{y}^\top}{1+\bm{y}^\top\bm{Ay}} \right\|_F\\
    &= \frac{\|\bm{y}\|^2_2}{|1+\bm{y}^\top\bm{Ay}|} \\
    &\leq \frac{\|\bm{y}\|^2_2}{\bm{y}^\top\bm{Ay}} \\
    &\leq \frac{1}{\lambda_{\min}(\bm{A})} < \zeta,
\end{align*}
where $\zeta$ is a constant as in \Cref{lemma:lammin}. 
The first inequality holds since $\bm{A}\coloneqq\bm{Z}^\top\bm{Z}$ is a positive semidefinite matrix with probability $1-\varepsilon$ by \Cref{lemma:lammin} and the second inequality follows from the property of Rayleigh quotient. The last inequality directly follows from \Cref{lemma:lammin}.
Hence, we have the following inequalities with probability at least $1-\varepsilon$:
\begin{align*}
    \|\varphi_{\tilde{\bm{T}}}(\bm{P}) - \varphi_{\bm{T}}(\bm{P})\|_F 
    \leq \mu U \qty(\eta + \zeta N),
\end{align*}
and
\begin{align*}
    \|\varphi_{\tilde{\bm{T}}}(\bm{P}) - \varphi_{\bm{T}'}(\bm{P})\|_F \leq \mu U \qty(\eta + \zeta N).
\end{align*}
Therefore, we have the following with probability at least $1-\varepsilon$:
\begin{align*}
    \left\|\varphi_{\bm{T}}(\bm{P}) - \varphi_{\bm{T}'}(\bm{P})\right\|_F &= \left\|\varphi_{\bm{T}}(\bm{P}) - \varphi_{\tilde{\bm{T}}}(\bm{P}) + \varphi_{\tilde{\bm{T}}}(\bm{P}) - \varphi_{\bm{T}'}(\bm{P})\right\|_F \\
    &\leq \left\|\varphi_{\tilde{\bm{T}}}(\bm{P}) - \varphi_{\bm{T}}(\bm{P}) \right\|_F + \left\|\varphi_{\tilde{\bm{T}}}(\bm{P}) - \varphi_{\bm{T}'}(\bm{P}) \right\|_F \\
    &\leq 2\mu U \qty(\eta + \zeta N).
\end{align*}
\end{proof}
\section{Proof of \Cref{lemma:P-gap}}\label{proof:lemma:p-gap}
\begin{proof}
Let $\delta^{(i)}\coloneqq\|\bm{P}^{(i)}-\bm{P}'^{(i)}\|_F$. We have $\delta^{(i)}=0$ for $i=1,\dots,k$. From Lemma \ref{lemma:sigma-bounded}, we have the following with probability at least $1-\varepsilon$:
\begin{align*}
    \delta^{(i+1)}
    &= \left\| \bm{P}^{(i+1)}-\bm{P}'^{(i+1)} \right\|_F \\
    &= \left\| \bm{P}^{(i+1)} -\bm{P}^{(i)} + \bm{P}^{(i)} - \bm{P}'^{(i)}+\bm{P}'^{(i)}-\bm{P}'^{(i+1)} \right\|_F \\
    &\leq \left\| \bm{P}^{(i+1)} -\bm{P}^{(i)} \right\|_F + \left\| \bm{P}'^{(i+1)} -\bm{P}'^{(i)} \right\|_F + \left\| \bm{P}^{(i)} -\bm{P}'^{(i)} \right\|_F \\
    &\leq 2\eta NU\mu + \delta^{(i)},
\end{align*}
for each $i=k,\dots,K$. Therefore, by using the above relation repeatedly and from Lemma \ref{lemma:growth}, we have the following with probability at least $1-\varepsilon$:
\begin{align*}
    \delta^{(K+1)} &\leq 2(K-k)\eta NU\mu + 2\mu U \qty(\eta + \zeta N )\\
    &= 2 \mu\eta U \qty{1 + \qty(K-k + \frac{\zeta}{\eta})N}.
\end{align*}
\end{proof}
\section{Proof of \Cref{prop:stability}}\label{appendix:proof_stability}
\begin{proof}
For any $\bm{t}\in\Delta^{M-1}$ and for any $\{\bm{T}_i\}^{K}_{i=1},\,\{\bm{T}'_i\}^{K}_{i=1}\subseteq (\Delta^{M-1})^{K}$ such that $\{\bm{T}_i\}^{K}_{i=1}$ and $\{\bm{T}'_i\}^{K}_{i=1} $ differs only one example, we have
\begin{align}
\begin{split}\label{eq:PPbound}
    |\ell(A(\bm{T});\bm{t}) - \ell(A(\bm{T}');\bm{t})|
    &=\left| \| \bm{b}(\bm{t}|\bm{P}^{(K+1)})-\bm{x}^\star(\bm{t})\|_2 - \| \bm{b}(\bm{t}|\bm{P}'^{(K+1)})-\bm{x}^\star(\bm{t}) \|_2 \right|\\
    &\leq \left\|\bm{b}(\bm{t}|\bm{P}^{(K+1)})-\bm{b}(\bm{t}|\bm{P}'^{(K+1)}) \right\|_2 \\
    &= \left\| \left(\bm{P}^{(K+1)}-\bm{P}'^{(K+1)}\right)^\top\bm{z}(\bm{t}) \right\|_2 \\
    &\leq \|\bm{z}(\bm{t})\|_2 \cdot \left\|\bm{P}^{(K+1)}-\bm{P}'^{(K+1)}\right\|_F,
\end{split}
\end{align}
where the first inequality follows from the reverse triangle inequality.
We can bound the right-hand side of \eqref{eq:PPbound} with probability at least $1-\varepsilon$ by \Cref{lemma:P-gap}. Since the left-hand side of \eqref{eq:PPbound} is bounded for all $\bm{t}\in\Delta^{M-1}$, we see that \Cref{alg:proposed} satisfies PAC uniform stability.
\end{proof}
\section{Problem Definition}\label{appendix:problem}
\paragraph{Scaled-MED} is a three-variable three-objective problem defined by:
\begin{align*}
    \mathrm{minimize}\quad &\bm{f}(\bm{x})\coloneqq(f_1(\bm{x}),f_2(\bm{x}),f_3(\bm{x}))^\top \\
    \mathrm{subject\:to}\quad &\bm{x}\in \mathbb{R}^3 \\
    \text{where}\quad
    &f_1(\bm{x}) = x^2_1 + 3(x_2-1)^2 + 2(x_3-1)^2, \\
    &f_2(\bm{x}) = 2(x_1-1)^2 + x_2^2 + 3(x_3-1)^2, \\
    &f_3(\bm{x}) = 3(x_1-1)^2 + 2(x_2-1)^2 + (x_3+1)^2. 
\end{align*}
\paragraph{Skew-$M$MED} is a $M$-variable $M$-objective problem defined by:
\begin{align*}
    \mathrm{minimize}\quad &\bm{f}(\bm{x})\coloneqq(f_1(\bm{x}),\dots,f_M(\bm{x}))^\top \\
    \mathrm{subject\:to}\quad &\bm{x}\in \mathbb{R}^M \\
    \text{where}\quad
    &f_m(\bm{x}) = \left( \frac{1}{\sqrt{2}} \|\bm{x}-\bm{e}_m\|^2 \right)^{p_m},\\
    & p_m = \exp(\frac{2(m-1)}{M-1} - 1), \\
    &\bm{e}_m = (0,\dots,0,\underbrace{1}_{\text{$m$th}},0,\dots,0)^\top,\\
    \text{for}\quad &m=1,\dots,M.
\end{align*}
\paragraph{Skew-$M$MMD} is a $M$-variable $M$-objective problem defined by:
\begin{align*}
    \mathrm{minimize}\quad &\bm{f}(\bm{x})\coloneqq(f_1(\bm{x}),\dots,f_M(\bm{x}))^\top \\
    \mathrm{subject\:to}\quad &\bm{x}\in X \subseteq \mathbb{R}^M \\
    \text{where}\quad
    &f_m(\bm{x}) = \|\bm{A}_m (\bm{x}-\bm{c}_m )\|^{p_m},\\
    & p_m > 0, \\
    \text{for}\quad &m=1,\dots,M.
\end{align*}
In the experiments in \Cref{subsec:gd}, we set $M=3$, $X=\mathbb{R}^3$, 
\begin{align*}
    A_1 \coloneqq \diag{\qty(\frac{3}{5},\frac{4}{5},\frac{4}{5})},\;
    A_2 \coloneqq \diag{\qty(\frac{4}{5},\frac{3}{5},\frac{4}{5})},\;
    A_3 \coloneqq \diag{\qty(\frac{4}{5},\frac{4}{5},\frac{3}{5})},
\end{align*}
$\bm{c}_m\coloneqq\bm{e}_m$ and $p_m\coloneqq\exp(\frac{2(m-1)}{M-1}-1)$. Note that $\diag{\qty(\cdot)}$ denotes the diagonal matrix.
\section{Analytical solution of scaled-MED}\label{appendix:scaledmed_opt}
We derive a map $\bm{x}^\star\colon\Delta^2\to X^{\star}(\bm{f})$ for scaled-MED. For any $\bm{t}=(t_1,t_2,t_3)\in\Delta^2$, the scalarizing function weighted by $\bm{t}$ is defined by
\begin{align*}
    f(\bm{x}|\bm{t}) &\coloneqq \sum_{m=1}^3 t_m f_m(\bm{x}) \\
    &= t_1x^2_1 + 2t_2(x_1-1)^2 + 3t_3(x_1-1)^2 \\
    &\quad + 3t_1(x_2-1)^2 + t_2x_2^2 + 2t_3(x_2-1)^2 \\
    &\quad + 2t_1(x_3-1)^2 + 3t_2(x_3-1)^2 + t_3(x_3+1)^2.
\end{align*}
Since $f(\bm{x}|\bm{t})$ is a convex quadratic function with respect to each $x_1,\,x_2$ and $x_3$, its optimal solution $\qty(x^\star_1(\bm{t}),x^\star_2(\bm{t}),x^\star_3(\bm{t}))^\top$ satisfies the following conditions:
\begin{align*}
    \eval{\pdv{f(\bm{x}|\bm{t})}{x_1}}_{\bm x=\bm x^\star(\bm{t})} &= 2t_1x_1 + 4t_2(x_1-1) + 6t_3(x_1-1) = 0, \\
    \eval{\pdv{f(\bm{x}|\bm{t})}{x_2}}_{\bm{x}=\bm{x}^\star(\bm{t})} &= 6t_1(x_2-1) + 2t_2x_2 + 4t_3(x_2-1) = 0, \\
    \eval{\pdv{f(\bm{x}|\bm{t})}{x_3}}_{\bm{x}=\bm{x}^\star(\bm{t})} &= 4t_1(x_3-1) + 6t_2(x_3-1) + 2t_3(x_3+1) = 0.
\end{align*}
By solving the above equation, the map $\bm{x}^\star(\bm{t})$ is given by
\begin{align*}
    \bm{x}^\star(\bm{t}) = \qty(x^\star_1(\bm{t}),x^\star_2(\bm{t}),x^\star_3(\bm{t}))^\top
    &= \qty(\frac{2t_2 + 3t_3}{t_1 + 2t_2 + 3t_3},\,\frac{3t_1 + 2t_3}{3t_1 + t_2 + 2t_3},\,\frac{2t_1 + 3t_2 - t_3}{2t_1 + 3t_2 + t_3})^\top.
\end{align*}

\end{document}